%% file: Steklov.tex
\documentclass[11pt,reqno,a4paper]{article}
\include{Preamble}
\author{T. V. Anoop\footnote{corresponding author and also supported by the INSPIRE Research Grant DST/INSPIRE/04/2014/001865.}\;,\;
Nirjan Biswas}

\title{On global bifurcation for the nonlinear Steklov problems}
\date{}
\begin{document}
\maketitle

\input{Abstract}


\input{Introduction}
\input{Preliminaries}

\input{Functional}

\input{Proofs}
\bibliographystyle{plainurl}
\bibliography{Steklov}

\noi {\bf  T. V.  Anoop } \\  Department of Mathematics, \\   Indian Institute of Technology Madras, \\ Chennai, 600036, India. \\ 
{\it Email}:{ anoop@iitm.ac.in}  \vspace{0.5 cm} \\
\noi {\bf  Nirjan Biswas } \\  Department of Mathematics, \\   Indian Institute of Technology Madras, \\ Chennai, 600036, India. \\ {\it Email}:{ nirjaniitm@gmail.com}

\end{document}

%% file: Preamble.tex
\usepackage{amsmath,amssymb,amsthm,amsfonts,epsfig,enumerate}
\usepackage{mathtools}
\usepackage{multicol}
\usepackage{mathrsfs,amssymb,soul}
\usepackage[utf8]{inputenc}
\usepackage{amsmath,amssymb,amsthm,amsfonts,epsfig,enumerate}
\usepackage{pst-plot,pst-node}
\usepackage{xfrac}

\allowdisplaybreaks[4] 

\usepackage{graphicx}
\usepackage{xcolor}
\usepackage{hyperref}
\hypersetup{
    colorlinks=true,
    linkcolor=myblue,
    filecolor=magenta,      
    urlcolor=mygreen,
    citecolor=mygreen,
}
\definecolor{mygreen}{rgb}{0.01,0.6,0.2}
\definecolor{myblue}{rgb}{0.01, 0.18, 1.0}
\newtheorem{definition}{Definition}[section]
\newtheorem{theorem}[definition]{Theorem}

\newtheorem{proposition}[definition]{Proposition}
\newtheorem{Lemma}[definition]{Lemma}

\theoremstyle{definition}
\newtheorem{remark}[definition]{Remark}
\newtheorem{example}[definition]{Example}
\numberwithin{equation}{section}
\DeclarePairedDelimiter\abs{\lvert}{\rvert}%
\DeclarePairedDelimiter\norm{\lVert}{\rVert}%
\makeatletter
\let\oldnorm\norm
\def\norm{\@ifstar{\oldnorm}{\oldnorm*}}
\newcommand{\al} {\alpha}

\newcommand{\pa} {\partial}

\newcommand{\de} {\delta}

\newcommand{\ga} {\gamma}

\newcommand{\Om} {\Omega}

\newcommand{\la} {\lambda}

\newcommand{\Gr} {\nabla}
\newcommand{\no} {\nonumber}
\newcommand{\noi} {\noindent}

\newcommand{\ep} {\epsilon}
\newcommand{\ra} {\rightarrow}

\newcommand\restr[2]{{

\left.\kern-\nulldelimiterspace 
#1 
\right|_{#2} 
}}

\def\wp{{W^{1,p}(\Om)}}

\def\wN{{W^{1,N}(\Om)}}

\def\C{{\mathcal C}}

\def\S{{\mathcal S}}

\def\R{{\mathbb R}}

\def\RN{{\mathbb R}^N}
\def\N{{\mathbb N}}
\def\F{{\mathcal F}}
\def\G{{\mathcal G}}
\def\({{\Big(}}
\def\){{\Big)}}

\def\c{{\C^1}}
\def\dis{{\displaystyle \int_{\Omega}}}

\def\p{{p^{\prime}}}

\def\dt{{\rm d}t}
\def\dx{{\rm d}x}
\def\dy{{\rm d}y}

\def\dJ{{\rm d}J}

\def\dsg{{\rm d}\sigma}
\def\pa{{\partial}}
\def\dispa{{\displaystyle \int_{\partial \Omega}}}
\def\hm{{\mathcal{H}}^{N-1}}
\def\pcr{{\frac{p(N-1)}{N-p}}}

\usepackage[hyperpageref]{backref}
\def\lfrac#1#2{{\scriptstyle\frac{#1}{#2}}}
\def\llfrac#1#2{{\scriptscriptstyle\frac{#1}{#2}}}
\DeclareMathOperator*\lowlim{\underline{lim}}
\DeclareMathOperator*\uplim{\overline{lim}}

%% file: Abstract.tex
\begin{abstract}
For $p \in (1, \infty),$ for an integer $N \geq 2$ and for a bounded  Lipschitz  domain $\Om$, we consider the following nonlinear Steklov bifurcation problem 
\begin{align*}
-\Delta_p \phi  = 0 \ \text{in}\ \Om, \quad \abs{\Gr \phi}^{p-2} \frac{\partial \phi}{\partial \nu} = \la \left( g \abs{\phi}^{p-2}\phi + f r(\phi) \right) \ \text{on} \ \pa \Om,
\end{align*}
where $\Delta_p$ is the $p$-Laplace operator, $g,f \in L^1(\pa \Om)$ are indefinite weight functions and $r \in C(\R)$ satisfies $r(0)=0$ and certain growth conditions near zero and at infinity.  For $f,g$ in some appropriate Lorentz-Zygmund spaces, we establish the existence of a continuum that bifurcates from $(\la_1,0)$, where $\la_1$ is the first eigenvalue of the following   nonlinear Steklov eigenvalue problem
\begin{align*}\label{Steklov abstract}
-\Delta_p \phi  = 0 \ \text{in}\ \Om, \quad \displaystyle \abs{\Gr \phi}^{p-2} \frac{\pa \phi}{\pa \nu} = \la g \abs{\phi}^{p-2}\phi  \ \text{on} \ \pa \Om.
\end{align*}
\end{abstract}

%% file: Introduction.tex
\section{Introduction}
Let $\Om$ be a bounded  Lipschitz domain in $\R^N$ $(N\ge 2)$ with the boundary  $\pa \Om$. For $p \in (1, \infty),$  we consider the following nonlinear Steklov bifurcation problem:  
\begin{equation}\label{Steklov pertub}
\begin{aligned}
         -\Delta_p \phi & = 0 \ \text{in}\ \Om,\\
           \displaystyle \abs{\Gr \phi}^{p-2}\frac{\pa \phi}{\pa \nu} &= \la \left( g |\phi|^{p-2}\phi + f r(\phi) \right)  \ \text{on} \ \pa \Om, 
\end{aligned}
\end{equation}
where $\Delta_p$ is the $p$-Laplace operator defined as $\Delta_p(\phi) = \text{div}(\abs{\Gr \phi}^{p-2} \Gr \phi),$ $f,g \in L^1(\pa \Om)$ are indefinite weights functions and  $r \in C(\R)$ satisfying $r(0) = 0$. A function $\phi \in \wp$ is said to be a solution of \eqref{Steklov pertub}, if
\begin{align}\label{weak pertub}
          \dis |\Gr \phi|^{p-2} \Gr \phi \cdot \Gr v\; \dx = \la \dispa \left( g \abs{\phi}^{p-2} \phi v + f r(\phi)v \right) \; \dsg, \quad \forall v \in \wp.
\end{align}
Since $r(0) = 0$, the set $\{ (\la,0): \la \in \R \}$ is always a trivial branch of solutions of \eqref{Steklov pertub}. We say a real number $\la$ is a bifurcation point of \eqref{Steklov pertub}, if there exists a sequence $\{(\la_n , \phi_n)\}$ of nontrivial weak solutions of \eqref{Steklov pertub} such that $\la_n \rightarrow \la$ and $\phi_n \rightarrow 0$ in $\wp$ as $n \rightarrow \infty.$

The bifurcation problem arises in numerous contexts in mathematical and engineering applications. For example, in reaction diffusion \cite{JA}, elasticity theory \cite{BRK, TW}, population genetics \cite{BT}, water wave theory \cite{Levi}, stability problems in engineering  \cite{Troger, Troger1}. Many authors considered the following nonlinear bifurcation problem with different boundary conditions
\begin{align}\label{pert 1}
          -\Delta_p \phi  =  \la g \abs{\phi}^{p-2} \phi + h(\la, x, \phi) \; \text{in} \; \Om,
  \end{align}
 where $h$ is assumed to be a Carath\'{e}odory function satisfying $h(\la, x, 0) = 0$. There are various sufficient conditions available on $g$ for the existence of a bifurcation point of \eqref{pert 1}. For Dirichlet boundary condition, $g = 1$  \cite {Drabek2, Girg, DelPino}, $g \in L^r(\Om)$  with $r > \frac{N}{2}$  \cite{AGJ}, $g \in L^{\infty}(\RN)$ \cite{Drabek-Huang}. There are few works that deal with $h$ of the form $\la f(x) r(\phi)$ with  continuous $r$ satisfying $r(0) = 0$ and certain growth condition at zero and at infinity, see  for $g, f$ in H\"{o}lder continuous spaces \cite{ Rumbos}, in certain Lebesgue spaces \cite{GLR}, in Lorentz spaces \cite{AMM, Lucia}. The bifurcation problem \eqref{pert 1} with Neumann boundary condition is considered for $g = 1$ in \cite{Drabek2}, for smooth $f, g$ in \cite{Brown}.

For $p=2$,  \eqref{Steklov pertub}  is considered in  \cite{Cushing, Cushing1, Stuart} for $g=1$ and  continuous $f,$ and in \cite{Pagani}  for $f, g \in L^{\infty}(\pa \Om)$. Indeed, there are many singular weights (not belonging to any of the Lebesgue spaces) that appear in problems in quantum mechanics, molecular physics, see \cite{FMT, Ferreira, Frank}. In this article, we enlarge the class of weight functions beyond the classical Lebesgue spaces. More precisely, we consider  $f,g$ in certain Lorentz-Zygmund spaces, and study the existence of bifurcation point for \eqref{Steklov pertub}.

Using the weak formulation, it is easy to see that \eqref{pert 1} is equivalent to the following 
 operator equation:
\begin{align}\label{operator eqn}
    A(\phi) =  \la G(\phi) + H(\la, \phi), \quad \phi\in X,
\end{align}
where $X$ is the Banach space $W^{1,p}(\Om)$ or $W^{1,p}_0(\Om)$ depending on the boundary conditions,  $A, G, H(\la,.): X\ra X'$ defined as $\left< A(\phi), v \right> = \int_{\Om} \abs{\Gr \phi}^{p-2} \Gr \phi \cdot \Gr v \, \dx;$ $\left< G(\phi), v \right> = \int_{\Om} g \abs{\phi}^{p-2} \phi v \, \dx;\,$ $  \left< H(\la, \phi), v \right> = \int_{\Om} h(\la, x, \phi) v \, \dx.$ For $p = 2$, $A$ is an invertible map. Using the Leray-Schauder degree \cite{Leray}, Krasnosel’skii in \cite{Krasnosel} gave sufficient conditions on $L=A^{-1}G,K=A^{-1}H$ so that, for any eigenvalue $\mu=\la^{-1}$ of $L$ with odd  multiplicity,  $(\la, 0)$ is a bifurcation point of \eqref{operator eqn}. Later, Rabinowitz \cite[Theorem 1.3]{Rabinowitz}, extended this result by  exhibiting  a continuum of nontrivial solutions of \eqref{operator eqn} bifurcating from  $(\la, 0)$ which is either unbounded in $\R\times X$ or meets at $(\la*, 0)$, where $\mu={\la*}^{-1}$ is an eigenvalue of $L.$ Further, if $\mu$ has multiplicity one, then  this continuum decompose into two subcontinua of nontrivial solutions of \eqref{operator eqn}, see  \cite{Ambrosetti, Dancer, Dancer1, Rabinowitz, Rabinowitz1}.  For $p \neq 2$,  the  Leray-Schauder degree is extended for certain maps between $X$ to $X'$ \cite{ Browder, Skrypnik} and then an analogue of  Rabinowitz result is proved  for the first eigenvalue of $A=\la G$, see \cite{Drabek2, Drabek-Huang, Girg, DelPino}.

To study the bifurcation problem  \eqref{Steklov pertub}, we consider the following  nonlinear eigenvalue problem:
\begin{equation}\label{Steklov weight}
\begin{aligned}
      -\Delta_p \phi & = 0 \ \text{in}\ \Om,\\
      \displaystyle \abs{\Gr \phi}^{p-2}\frac{\pa \phi}{\pa \nu} &= \la g |\phi|^{p-2}\phi  \ \text{on} \ \pa \Om. 
\end{aligned}
\end{equation}
 For $N=2$, $p=2$ and $g = 1,$ the problem \eqref{Steklov weight} is first considered by Steklov in \cite{Stekloff}. A real number $\la $ is said to be an eigenvalue of \eqref{Steklov weight}, if there exists $\phi \in \wp \setminus \{ 0 \}$ satisfying the following weak formulation 
\begin{align}\label{Sw weak}
\dis |\Gr \phi|^{p-2} \Gr \phi \cdot \Gr v\; \dx = \la \dispa  g \abs{\phi}^{p-2}\phi v \; \dsg, \quad \forall v \in \wp.
\end{align} 
For $N>p$, the classical trace embeddings (\cite[Theorem 4.2  and Theorem 6.2]{Nevcas}) gives
$$\wp \hookrightarrow L^{q}(\pa \Om), \text{ where }  q \in \left[1, \pcr \right],$$ and for  $q<\pcr$ the above embedding is compact. Thus, by the H\"{o}lder inequality  the right hand side of \eqref{Sw weak} is finite for $g\in L^r(\pa\Om)$ with $r\in \left[\frac{N-1}{p-1},\infty\right]$ and for any $\phi, v \in \wp.$ We say an eigenvalue $\la$ is principal, if there exists an eigenfunction  of \eqref{Steklov weight} corresponding to $\la$ that does not change it's sign in $\overline{\Om}.$ Notice that, zero is always a principal eigenvalue of \eqref{Steklov weight} and  if $\int_{\pa \Om} g \geq 0$, then zero is the only principal eigenvalue. Thus for the existence of a positive principal eigenvalue of \eqref{Steklov weight}, it is necessary to have a $g$ satisfying  $\int_{\pa \Om} g < 0$ and the $(N-1)$-dimensional Hausdorff measure of $\text{supp}(g^+)$ is nonzero. In \cite{Torne}, for $g\in L^r(\pa\Om)$ with $r\in \left(\frac{N-1}{p-1},\infty\right]$ satisfying the above necessary conditions, with the help of the above compact embedding, the authors proved the existence of a positive principal eigenvalue of \eqref{Steklov weight}. For $N=p$, $\wp$ is embedded compactly in $L^{q}(\pa \Om)$ for  $q \in [1, \infty).$ Thus  for $g \in L^r(\pa \Om)$ with $r\in (1,\infty]$ satisfying the above necessary condition, \eqref{Steklov weight} admits a positive principal eigenvalue, as obtained in \cite{Torne}.

In order to enlarge the class of weight functions beyond $L^r$, we use the trace embeddings due to Cianchi-Kerman-Pick. In \cite{CianchiPick}, the authors  improved the classical trace embeddings by providing finer trace embeddings as below:
\begin{align*}
   & (i) \; \text{For} \; N > p: \quad \wp \hookrightarrow L^{\frac{p(N-1)}{N-p},p}(\pa \Om)\subsetneq L^{\frac{p(N-1)}{N-p}}(\pa\Om). \\
   & (ii) \; \text{For} \; N = p: \quad \wp \hookrightarrow L^{\infty,N;-1}(\pa \Om)\subsetneq L^{q}(\pa \Om), \;\; \forall\, q \in [1, \infty).
\end{align*}
Nevertheless, none of these embeddings are compact. In this article, we use the above trace embeddings and  prove the existence of a positive principal eigenvalue of \eqref{Steklov weight} for weight functions  in certain Lorentz-Zygmund spaces. More precisely, for $1 \leq d < \infty,$ we consider the following closed subspaces: 
\begin{equation*}
\begin{aligned}
      &\F_{d}  := \text{closure of} \; \c(\pa \Om) \; \text{in the Lorentz space} \; L^{d,\infty}(\pa \Om),\\
      &\G_{d}  := \text{closure of} \; \c(\pa \Om) \; \text{in the Lorentz-Zygmund space} \; L^{d, \infty;N}(\pa \Om). 
\end{aligned}
\end{equation*}

\begin{theorem}\label{Steklov existence}
Let $p \in (1, \infty)$ and $N \geq p.$ Let $g^+ \not \equiv 0$, $\int_{\pa \Om} g  < 0$ and  
$$g \in \left\{\begin{array}{ll}
                 \F_{\frac{N-1}{p-1}} & \text{  for } N > p,\\
                 \G_1 & \text{  for } N = p.
          \end{array}\right.$$ 
Then $$\la_1 = \inf \left\{ \displaystyle \int_{\Om} |\Gr \phi|^p : \phi \in \wp, \dispa g |\phi|^p = 1 \right\} $$ is the unique positive principal eigenvalue of \eqref{Steklov weight}. Furthermore, $\la_1$ is simple and isolated.
\end{theorem}

Indeed,  $L^{\frac{N-1}{p-1}}(\pa \Om)$ is contained in $\F_{\frac{N-1}{p-1}}$ (for $N > p$) and  $L^q(\pa\Om)$ (for $q > 1$) is contained in $\G_1$ (for $N=p$) (see Remark \ref{Stricly contained}). Thus the above theorem extends the result of \cite{Torne}.

Having obtained the right candidate for bifurcation point, we can study \eqref{Steklov pertub} for weights in appropriate Lorentz-Zygmund spaces. For this, let us consider the following set:
\begin{align*}
    \S = \left\{ (\la, \phi) \in \R \times \wp: (\la, \phi) \; \text{is a solution of} \; \eqref{Steklov pertub} \; \text{and} \; \phi \not \equiv 0 \right\}.
\end{align*}
We say $\C \subset \S$ is a continuum of nontrivial solutions of \eqref{Steklov pertub} if it is connected in $\R \times \wp.$  In this article, we prove the existence of a continuum $\C$ of nontrivial solutions of \eqref{Steklov pertub} that bifurcates from $(\la_1, 0)$.

For $p \in (1, \infty)$ and $g$ as in Theorem \ref{Steklov existence},  depending on the dimension  we make  the following  assumptions on  $r$ and $f$:
\begin{equation*}
({\bf H1}) \left\{ \begin{aligned}
&{\bf(a)} \quad \displaystyle\lim_{\abs{s} \rightarrow 0}\frac{\abs{r(s)}}{\abs{s}^{p-1}}=0 \; \text{ and} \;\abs{r(s)} \leq C\abs{s}^{\gamma - 1} \; \text{for some} \; \gamma\in \left(1, \frac{p(N-1)}{N-p} \right). \\
&{\bf(b)} \quad g\in \F_{\frac{N-1}{p-1}}, \; f \in \left\{\begin{array}{ll} 
                                                              \F_{\tilde{p}}, & \text {if }  \gamma \geq p, \; \text{where} \; \displaystyle \frac{1}{\tilde{p}} + \frac{\gamma(N-p)}{p(N-1)} = 1; \\  \F_{\frac{N-1}{p-1}}, & \text{if} \; \ga < p.  \\
                                                           \end{array} \right. 
                \end{aligned} \right.
\end{equation*}
\begin{equation*}
({\bf H2}) \left\{ \begin{aligned}
                     &{\bf(a)} \quad \displaystyle \lim_{\abs{s} \rightarrow 0}\frac{\abs{r(s)}}{\abs{s}^{N-1}}=0 \; \text{ and} \;  \abs{r(s)}  \leq  C\abs{s}^{\gamma-1} \; \text{for some} \; \gamma\in (1, \infty). \quad \quad \quad \quad\\
                     &{\bf (b)} \quad g\in \G_1, \; f  \in \G_d \; \text{with} \; d > 1.
                     \end{aligned} \right.
\end{equation*}

\begin{theorem}\label{bifur}
Let $p \in (1, \infty)$. Assume that  $ r,g \text{ and } f$  satisfy 
   ({\bf H1}) for $N>p$ and satisfy ({\bf H2}) for $N=p$. Then $\la_1$ is a bifurcation point of \eqref{Steklov pertub}. Moreover, there exists a continuum of nontrivial solutions $\C$ of \eqref{Steklov pertub} such that $(\la_1, 0) \in \overline{\C}$ and either 
   \begin{enumerate}[(i)]
       \item $\C$ is unbounded, or \item $\C$ contains the point $(\la, 0)$, where $\la$ is an eigenvalue of \eqref{Steklov weight} and $\la \neq \la_1$.
   \end{enumerate}
\end{theorem}

The rest of the article is organized as follows. In Section 2, we give the definition and list some properties  of symmetrization and Lorentz-Zygmund spaces. We also state the classical trace embedding theorems and their refinements. The definition and some of the properties of degree of a certain class of  nonlinear maps between $\wp$ and $(\wp)'$ are also given in this section. In Section 3, we develop a functional framework associated with our problem and prove many  results that we needed to prove our main theorems. Section 4 contains the proofs of Theorem 1.1 and Theorem 1.2.

%% file: Preliminaries.tex
\section{Preliminaries}
In this section, we briefly describe the one-dimensional decreasing rearrangement with respect to $(N-1)$-dimensional Hausdorff measure. Using this,  we define Lorentz-Zygmund spaces over the  boundary and give examples of functions in these spaces. Further, we state  the classical trace embeddings of $\wp$, and it's refinements due to Cianchi et al. We also define the degree for a certain class of nonlinear maps and list  some of the results that we use in this article.

\subsection{Symmetrization}
Let $ \Om \subset \R^N $ be a  bounded Lipschitz domain. Let  $ \mathcal{M}(\pa \Om)$ be the collection of all real valued $(N-1)$-dimensional Hausdorff measurable functions defined on $\pa \Om$. Given a function $f \in \mathcal{M}(\pa \Om),$ and for $s > 0, $ we define $E_{f}(s) = \{ x \in \pa \Om : |f(x)| > s \}.$  The \textit{distribution function} $ \alpha_{f} $ of $f$ is defined as $\alpha_{f}(s) = \hm(E_{f}(s)) \; \text{for} \; s > 0.$
We define the \textit{one dimensional decreasing rearrangement} $ f^{*} $ of $f$ as
\begin{align*}
       f^*(t) = \inf \left\{ s > 0 : \alpha_{f}(s) < t \right\}, \; \mbox{ for } t > 0. 
\end{align*}
The map $f \mapsto f^*$ is not sub-additive. However, we obtain a sub-additive function from $f^*,$ namely the maximal function $f^{**}$ of $f^*$, defined by 
\begin{equation*}\label{maximal}
    f^{**}(t)=\frac{1}{t}\int_0^tf^*(\tau)\, {\rm d}\tau, \quad t>0.
\end{equation*}
Next we state one important inequality concerning the symmetrization \cite[Theorem 3.2.10]{EdEv}.

\begin{proposition}\label{HL}(Hardy-Littlewood inequality)
Let $N \geq 2$ and let  $\Om$ be a   bounded Lipschitz domain in $\R^N$. Let $f$ and $g$ be nonnegative measurable functions defined on $\pa \Om$. Then
$$\int_{ \pa \Omega} fg \; \dsg \leq  \int_0^{\hm(\pa \Om)}f^*(t) g^*(t) \;\dt.$$
\end{proposition}

\subsection{Lorentz-Zygmund space}
The Lorentz-Zygmund spaces are three parameter family of function spaces that refine the classical Lebesgue spaces. For more details on Lorentz-Zygmund spaces, we refer to \cite{ColinRudnik, ET}. Here we consider the Lorentz-Zygmund spaces over $\pa \Om$  of a bounded domain $\Om$.

Let $\Om \subset \R^N$ be a bounded Lipschitz  domain. Let $f \in \mathcal{M}( \pa \Omega)$ and let $l_1(t) = 1 + \abs{\log(t)}$. For $(p,q, \al) \in [1,\infty]\times[1,\infty]\times \R$, consider the following quantity:
\begin{align*} 
       |f|_{(p,q; \al)} & := \norm{t^{\frac{1}{p}-\frac{1}{q}} {l_1(t)}^{\al} f^{*}(t)}_{{L^q((0,\hm(\pa \Om)))}} \\
                    & = \left\{\begin{array}{ll}
                              \left(\displaystyle\int_0^{\hm(\pa \Om)} \left[t^{\frac{1}{p}} {l_1(t)}^{\alpha} {f^{*}(t)} \right]^q \frac{\dt}{t}  \right)^{\frac{1}{q}}, &\; 1\leq q < \infty; \\ 
                              \displaystyle\sup_{0 < t < \hm(\pa \Om)} t^{\frac{1}{p}} {l_1(t)}^{\alpha} {f^{*}(t)}, &\; q=\infty.
                            \end{array}\right.
 \end{align*}
The Lorentz-Zygmund space $L^{p,q;\al}(\pa \Om)$ is defined as
\[ L^{p,q;\al}(\pa \Om) := \left \{ f\in \mathcal{M}( \pa \Om): \,   |f|_{(p,q;\al)}<\infty \right \},\]
where $ |f|_{(p,q;\al)}$ is a complete quasi norm on $L^{p,q;\al}(\pa \Om).$ For $p > 1,$ 
$$
\norm{f}_{(p,q, \al)} = \norm{t^{\frac{1}{p}-\frac{1}{q}} {l_1(t)}^{\al} f^{**}(t)}_{{L^q((0,\hm(\pa \Om)))}}
$$ 
is a norm in $L^{p,q;\al}(\pa \Om)$ which is equivalent to $\abs{f}_{(p,q, \al)}$ \cite[Corollary 8.2]{ColinRudnik}.  In particular,  $L^{p,q;0}(\pa \Om)$ coincides with the Lorentz space $L^{p,q}(\pa \Om)$ introduced by Lorentz in \cite{Lorentz}. In the following proposition we discuss some important properties of the Lorentz-Zygmund spaces that we will use in this article.

\begin{proposition}\label{properties}
Let $p,q,r,s \in [1, \infty]$ and $\al, \beta \in (-\infty, \infty).$
\begin{enumerate}[(i)]
      \item  Let $p \in (1, \infty)$. If $f \in L^{\infty, p; -1}(\pa \Om)$, then $\abs{f}^p \in L^{\infty,1; -p}(\pa \Om)$.  Moreover, there exists $C >         0$ such that $$ \norm{\abs{f}^p}_{(\infty,1; -p)} \leq C \norm{f}^p_{(\infty,p; -1)}. $$
      \item  Let $p \in (1, \infty)$. Then the space $ L^{1, \infty;p}(\pa \Om)$ is contained in the dual space of $ L^{\infty,1;-p}(\pa \Om)$.
      \item If $ r > p,$ then $L^{r,s; \beta}(\pa \Om) \hookrightarrow L^{p,q; \al}(\pa \Om)$, i.e., there exists a constant $C > 0$ such that
            \begin{align}\label{LZ3}
                 \norm{f}_{(p,q,\al)} \leq C \norm{f}_{(r,s,\beta)}, \quad \forall f \in L^{r,s; \beta}(\pa \Om).
            \end{align} 
      \item If either $q \leq s$ and $\al \geq \beta$ or, $q > s$ and $\al + \frac{1}{q} > \beta + \frac{1}{s},$ then $L^{p,q; \al}(\pa \Om) \hookrightarrow L^{p,s;\beta}(\pa \Om)$, i.e.,  there exists $C > 0$ such that 
      \begin{align}\label{LZ4}
                 \norm{f}_{(p,s;\beta )} \leq C \norm{f}_{(p,q; \al)}, \quad \forall f \in L^{p,q; \al}(\pa \Om).
            \end{align} 
             
      \item For $p \in (1, \infty)$, $ L^p(\pa \Om) \hookrightarrow L^{1, \infty;\alpha}(\pa \Om)$.
\end{enumerate}
\end{proposition}

\begin{proof} 

\noi (i) If $f \in L^{\infty, p; -1}(\pa \Om)$, then $\abs{f}_{(\infty, p; -1)} < \infty.$ Hence using $(|f|^p)^{*} = ( f^{*} )^p,$ we get 
\begin{align*}
     \abs{ |f|^p }_{(\infty, 1;  -p)} =  \int_{0}^{\hm(\pa \Om)}  \frac{(|f|^p)^{*}}{ {(l_1(t))}^p} \; \frac{dt}{t} = \left( \left(    \int_{0}^{\hm(\pa \Om)} \left( \frac{f^{*}(t)}{l_1(t)} \right)^p \; \frac{dt}{t} \right)^{\frac{1}{p}} \right)^p = \abs{ f }^p_{(\infty, p; -1)}.
\end{align*}
Therefore, $\abs{f}^p \in L^{\infty, 1;  -p}(\pa \Om).$ Now by the equivallence of norms, there exists $C_1, C_2 > 0$ such that 
\begin{align*}
\norm{ |f|^p }_{(\infty, 1;  -p)} \le C_1 \abs{ f }^p_{(\infty, p; -1)} \le C_1 C_2  \norm{ f }^p_{(\infty, p; -1)}.  
\end{align*}
Thus there exists $C > 0$ such that $\norm{ |f|^p }_{(\infty, 1;  -p)} \leq C \norm{ f }^p_{(\infty, p; -1)}.$ \\
\noi (ii) Let $f \in L^{\infty,1;-p}(\pa \Om)$ and $g \in L^{1, \infty;p}(\pa \Om)$. Then using the Hardy-Littlewood inequality (Proposition \ref{HL}), 
\begin{align*}
\dispa fg \; \dsg &\leq \int_0^{\hm(\pa \Om)}f^*(t) g^*(t)\; \dt  \\
&\leq \left( \sup_{0 < t < \hm(\pa \Om)}{t  g^{**}(t)(l_1(t))^p} \right) \left( \int_0^{\hm(\pa \Om)}  \frac{f^{**}(t)}{l_1(t))^p}  \; \frac{dt}{t} \right) \\
& = \norm{g}_{(1, \infty;p)}\norm{f}_{(\infty,1;-p)}. 
\end{align*}
Thus $f$ is in the dual space of $L^{1, \infty;p}(\pa \Om)$. \\
\noi (iii) Follows from \cite[Theorem 9.1]{ColinRudnik}. (iv) Follows from \cite[Theorem 9.3]{ColinRudnik}. \\
\noi (v) Let $f \in L^p(\pa \Om).$ Since $p > 1$, using \eqref{LZ3} there exists $C > 0$ such that 
$$\norm{f}_{(1, \infty;\alpha)} \leq C \norm{f}_{L^p(\pa \Om)}.$$ Therefore, $ L^p(\pa \Om)$ is continuously embedded into $L^{1, \infty;\alpha}(\pa \Om)$.
\end{proof}

The following characterization of the function space $\G_d$ follows by similar arguments as in the proof of \cite[Theroem 16]{Anoop}.

\begin{proposition}\label{G_d}
Let $N \geq 2$ and $d \in [1, \infty).$ Then $f \in \G_d$ if and only if 
$$
   \underset{t \rightarrow 0}{\lim} \; t^{\frac{1}{d}}(l_1(t))^N f^*(t) = 0.
$$ 
\end{proposition}

Next we list some properties of the Lorentz spaces. For more details on Lorentz spaces, we refer to \cite{Adams, EdEv, Hunt}.

\begin{proposition}\label{Lorentz properties}
 Let $p,q,r \in [1, \infty]$.
 \begin{enumerate}[(i)]
    \item Generalized H\"{o}lder inequality: Let $f \in L^{p_1, q_1}(\pa \Om)$ and $g \in L^{p_2, q_2}(\pa \Om)$, where  $(p_i, q_i) \in (1, \infty) \times       [1, \infty]$ for $i = 1,2$. If $(p,q)$ be such that $\frac{1}{p} = \frac{1}{p_1} + \frac{1}{p_2}$ and 
          $\frac{1}{q} = \frac{1}{q_1} + \frac{1}{q_2},$ then 
          \begin{align*}
                \norm{fg}_{(p,q)} \leq C \norm{f}_{(p_1,q_1)} \norm{g}_{(p_2,q_2)},
           \end{align*}
          where  $C = C(p) > 0$ is a constant such that $C = 1,$ if $p=1$ and $C = \p,$ if $p > 1$.
    \item  For $r>0$, $\norm{\abs{f}^{r}}_{\left(\frac{p}{r}, \frac{q}{r} \right)} = \norm{f}^{r}_{(p,q)}.$ 
    \end{enumerate}
\end{proposition}

\begin{proof} Proof of (i) follows from \cite[Theorem 4.5]{Hunt}. For $\al = 0,$ proof of (ii) directly follows from the definition of the Lorentz-Zygmund space.
\end{proof}

In the following we list some properties of the function space $\F_d.$

\begin{proposition}\label{F_d}
Let  $d,q \in (1, \infty).$ Then
\begin{enumerate}[(i)] 
    \item $L^{d,q}(\pa \Om) \subset \F_{d}$.
    \item Let $h \in L^{d, \infty}(\pa \Om)$ and $h > 0$. Let $f \in L^1(\pa \Om)$. If  $\int_{\pa \Om} h^{d-q} \abs{f}^q < \infty$ for $q \ge d$, then $f \in L^{d, q}(\pa \Om)$ and hence $f \in \F_{d}.$
    \item $f \in \F_{d}$ if and only if 
           \begin{align*}
                \lim_{t \rightarrow 0}t^{\frac{1}{d}} f^*(t) = 0 = \lim_{t \rightarrow \hm(\pa \Om)}t^{\frac{1}{d}} f^*(t).
            \end{align*}
\end{enumerate}
\end{proposition}

\begin{proof}
\noi (i) Using \eqref{LZ4} for $\al = \beta = 0$ and by the density arguments, we get $L^{d,q}(\pa \Om) \subset \F_{d}$. \\
\noi (ii) The result is obvious for $q = d$. For $q > d$, set $g = h^{\frac{d}{q} - 1} \abs{f}$. Then $g \in L^q(\pa \Om).$ Using Proposition \ref{Lorentz properties}, $h^{1- \frac{d}{q}} \in L^{\frac{dq}{q-d}, \infty}(\pa \Om).$ Therefore, by the generalized H\"{o}lder inequality (Proposition \ref{Lorentz properties}), $f \in L^{d, q}(\pa \Om).$ \\
\noi (iii) Follows by the similar arguments as in \cite[Theorem 3.3]{AMM}.
\end{proof}

\subsection{Examples}
Now we give some examples of functions in the Lorentz-Zygmund spaces that are defined on $\pa \Om$  of a Lipschitz bounded domain $\Om$.

\begin{example}\label{Ex1}
For $\Om = \{ (x,y) \in \R^2 : x^2 + y^2 < 1 \}$, we consider 
\begin{equation*} 
    g_1(x,y) = \abs{y}^{-\frac{1}{2}}, \quad \forall \; (x,y) \in \pa \Om.
\end{equation*}
For $s> 0,$ we can compute
\begin{equation*} 
 \begin{aligned}
      \al_{g_1}(s) = \left\{\begin{array}{ll} 
                               2\pi, & \text {for} \quad 0<s<1,\\ 
                               4 \,\sin^{-1}(\frac{1}{s^2}), & \text{for} \quad s \geq 1.\\
                          \end{array} \right.
\end{aligned}
\end{equation*} 
Thus $g_1^*(t) = \left( \text{cosec} \left( \frac{t}{4} \right) \right)^{\frac{1}{2}}$. Therefore, 
$$ 
\underset{0 < t < 2 \pi}{\sup} t^{\frac{1}{2}} \left( \text{cosec} \left( \frac{t}{4} \right) \right)^{\frac{1}{2}} < \infty; \quad  \underset{0 < t < 2 \pi}{\sup} t (l_1(t))^2 \left( \text{cosec} \left( \frac{t}{4} \right) \right)^{\frac{1}{2}} < \infty.
$$
Hence  $g_1 \in L^{2, \infty}(\pa \Om)$ and $g_1 \in L^{1, \infty;2}(\pa \Om).$ Furthermore, 
\begin{align*}
       \lim_{t \rightarrow 0} t^{\frac{1}{2}} \left( \text{cosec} \left( \frac{t}{4} \right) \right)^{\frac{1}{2}} > 0; \quad \lim_{t \rightarrow 0} t (l_1(t))^2 \left( \text{cosec} \left( \frac{t}{4} \right) \right)^{\frac{1}{2}} > 0.
\end{align*}
Hence  $g_1 \not \in \F_2$ (by Proposition \ref{F_d}) and $g_1 \not \in \G_1$ (by Proposition \ref{G_d}).
\end{example}

\begin{example}\label{Ex2}
Let $p \in (1, \infty)$ and $N > p$. For $0 < R < \frac{1}{2}$, let $$ \Om = \left\{ (x_1, x_2, \cdot, \cdot, \cdot, x_N) \in \R^N: |x_i| < R    \; (\text{for} \; i = 1, \cdot, \cdot, \cdot, N-1), 0 < x_N < 2R \right\} $$  and $A= \left\{ (x_1, x_2, \cdot, \cdot, \cdot, x_{N-1}, 0): |x_i| < R \right\}$. Now consider 
\begin{equation*} 
 \begin{aligned}
     g_2(x) = \left\{\begin{array}{ll} 
                            \abs{x_1 \log(\abs{x_1})}^{-\frac{p-1}{N-1}}, & \text {for} \quad x \in A,\\ 
                             0, & \text{for} \quad x \in \pa \Om \setminus A.\\
                          \end{array} \right. 
 \end{aligned}
\end{equation*}
Clearly $g_2 \in L^1(\pa \Om)$ and $g_2 \not \in L^r(\pa \Om)$ for $r \in \Big[ \frac{N-1}{p-1}, \infty \Big).$  Let
\begin{equation*} 
 \begin{aligned}
     h(x) = \left\{\begin{array}{ll} 
                            \abs{x_1}^{-\frac{p-1}{N-1}}, & \text {for} \quad x \in A,\\ 
                             0, & \text{for} \quad x \in \pa \Om \setminus A.\\
                          \end{array} \right. 
 \end{aligned}
\end{equation*}
We calculate $\al_h(s) = 2^{N-1}R^{N-2} s^{-\frac{N-1}{p-1}}$ and  $h^*(t) = (2^{N-1}R^{N-2})^{\frac{p-1}{N-1}} t^{-\frac{p-1}{N-1}}.$
Therefore,  $h \in L^{\frac{N-1}{p-1},\infty}(\pa \Om)$. For $q = \frac{N}{p-1},$ 
\begin{equation*} 
\begin{aligned}
     h^{\frac{N-1}{p-1} - q}(x) = \left\{\begin{array}{ll} 
                            \abs{x_1}^{\frac{1}{N-1}}, & \text {for} \quad x \in A,\\ 
                             0, & \text{for} \quad x \in \pa \Om \setminus A.\\
                          \end{array} \right. 
 \end{aligned}
\end{equation*}
Further, 
\begin{align*}
    \int_{\pa \Om} h^{\frac{N-1}{p-1} - q} g_2^q \, \dsg  = 2^{N-1} R^{N-2} \int_0^{R} t^{-1} \abs{\log(t)}^{- \frac{N}{N-1}} \, \dt < \infty.
\end{align*}
Therefore, by Proposition \ref{F_d}, $g_2 \in L^{\frac{N-1}{p-1},q }(\pa \Om)$  and hence $g_2 \in \F_{\frac{N-1}{p-1}}.$ 
\end{example}

\begin{example}\label{Ex3}
For $0 < R < 1$, let $\Om$ and $A$ be given as in the above example. For $q \in (1, \infty),$ we consider 
\begin{equation*} 
 \begin{aligned}
g_3(x) = \left\{\begin{array}{ll} 
                        \abs{x_1}^{- \frac{1}{q}}, & \text {for} \quad x \in A,\\ 
                         0, & \text{for} \quad x \in \pa \Om \setminus A.\\
               \end{array} \right. 
     \end{aligned}
\end{equation*}
Clearly $g_3 \not \in L^q(\pa \Om)$ for $q \in (1, \infty).$ Further, we calculate $\al_{g_3}(s) = 2^{N-1}R^{N-2} s^{-q}$ and  
$g^*_3(t) = (2^{N-1}R^{N-2})^{\frac{1}{q}}t^{- \frac{1}{q}}.$ Moreover, 
$$ 
\underset{t \rightarrow 0}{\lim} \; t^{\frac{q-1}{q}} (1+ |\log(t)|)^N = 0 
$$ 
and hence  $g_3 \in \G_1$ (by Proposition \ref{G_d}).
\end{example}

 \subsection{Trace embeddings}
Now we state the trace embeddings that play a vital role in this article. First, we state the classical trace embeddings to the Lebesgue spaces \cite[Theorem 4.2, Theorem 4.6, Theorem 6.2]{Nevcas}.

\begin{proposition}[Classical trace embeddings]\label{classical}
Let $N \geq 2$ and let  $\Om$ be a Lipschitz  bounded domain in $\R^N$. Let $p\in (1,\infty)$. Then the following embeddings hold:
\begin{enumerate}[(i)]
    \item If $N > p$ and  $q \in \left[1, \pcr \right]$, then $\wp \hookrightarrow L^{q}(\pa \Om),$ i.e., there exists $C = C(N,p) > 0$ satisfying 
           \begin{align*}
                \|\phi \|_{L^{q}(\pa \Om)} \leq C \| \phi \|_{\wp}, \quad \forall \phi \in \wp.   
            \end{align*}
          If $q \neq \pcr,$ then the above embedding is compact.
    \item If $N=p$ and  $q \in \left[ 1, \infty \right)$, then $\wp \hookrightarrow L^q(\pa \Om),$ i.e., there exists $C = C(N) > 0$ satisfying 
            \begin{align*}
                  \|\phi \|_{L^q(\pa \Om)} \leq C \| \phi \|_{\wp}, \quad \forall \phi \in \wp,   
            \end{align*}
           and the above embedding is compact.
\end{enumerate}
\end{proposition}

The following embeddings are due to Cianchi et al. \cite[Theorem 1.3]{CianchiPick} that extends the classical trace embeddings to the Lebesgue spaces with the finer embeddings to the Lorentz-Zygmund spaces.

\begin{proposition}[Finer trace embeddings]\label{Cianchi}
Let $N \geq 2$ and let  $\Om$ be a Lipschitz  bounded domain in $\R^N$. Let $p\in (1,\infty)$. Then the following embeddings hold: 
\begin{enumerate}[(i)]
    \item If $N > p$, then $\wp \hookrightarrow L^{\frac{p(N-1)}{N-p},p}(\pa \Om),$ i.e., there exists $C = C(N,p) > 0$ such that 
               \begin{align*}
                   \|\phi \|_{\left(\frac{p(N-1)}{N-p},p\right)} \leq C \| \phi \|_{\wp}, \quad \forall \phi \in \wp.    
               \end{align*}
    \item If $N=p$, then $\wp \hookrightarrow L^{\infty,N;-1}(\pa \Om)$, i.e., there exists $C = C(N) > 0$ such that
                \begin{align*}
                   \|\phi \|_{(\infty,N;-1)} \leq C \| \phi \|_{\wp}, \quad \forall \phi \in \wp.    
                \end{align*}
\end{enumerate}
\end{proposition}

The above finer trace embeddings help us to get the weighted trace inequality for a class of weight functions defined on the boundary.

\begin{proposition}\label{Hardy boundary}
(i) Let $N > p$ and $g \in L^{\frac{N-1}{p-1}, \infty}(\pa \Om)$. Then there exists  a constant $C = C(N,p) > 0$ satisfying \begin{align}\label{N>p}
         \dispa \abs{g} \abs{\phi}^p \leq C \norm{g}_{\left(\frac{N-1}{p-1}, \infty \right)} \norm{\phi}^p_{\wp}, \quad \forall \phi \in \wp.   
\end{align}
\noi (ii) Let $N =p$ and $g \in L^{1, \infty;N}(\pa \Om)$. Then there exists a constant $C = C(N) > 0$ satisfying 
\begin{align}\label{N=p}
          \dispa \abs{g} \abs{\phi}^p \leq C \norm{g}_{(1, \infty;N)} \norm{\phi}^p_{\wp}, \quad \forall \phi \in \wp.
\end{align}
\end{proposition}

\begin{proof}
(i) For $\phi \in \wp$, by the generalized H\"{o}lder inequality (Proposition \ref{Lorentz properties}) and Proposition \ref{Lorentz properties},  we obtain  
\begin{align*}
        \dispa \abs{g} \abs{\phi}^p \leq  \norm{g}_{\left(\frac{N-1}{p-1}, \infty \right)} \norm{\abs{\phi}^p}_{\left(\frac{N-1}{N-p},1 \right)}   =  \norm{g}_{\left( \frac{N-1}{p-1}, \infty \right)} \norm{\phi}^p_{\left(\frac{p(N-1)}{N-p},p \right)}. 
\end{align*}
Now using the finer trace embeddings (Proposition \ref{Cianchi}), we get 
\begin{align*}
        \dispa \abs{g} \abs{\phi}^p \leq C \norm{g}_{\left(\frac{N-1}{p-1}, \infty \right)} \norm{\phi}_{\wp}, \quad \forall \phi \in \wp,    
\end{align*}
where $C = C(N,p)$ is the embedding constant.

\noi (ii) For $\phi \in \wN$, using Proposition \ref{properties}, we obtain
\begin{align*}
        \dispa \abs{g} \abs{\phi}^N \leq  \norm{g}_{(1, \infty;N)} \norm{\abs{\phi}^N}_{(\infty, 1;-N)} & \leq C  \norm{g}_{(1, \infty;N)} \norm{\phi}^N_{(\infty, N;-1)}.
\end{align*}
Again using the finer trace embeddings, 
\begin{align*}
         \dispa \abs{g} \abs{\phi}^N \leq C  \norm{g}_{(1, \infty;N)} \norm{\phi}^N_{\wN}, \quad \forall \phi \in \wN,
\end{align*}
where $C = C(N) > 0$ is the embedding constant given in Proposition \ref{Cianchi}.
\end{proof}

\subsection{Degree}
We define the degree for certain class of maps from $\wp$ to it's dual $(\wp)'$. For more details on this topic, we refer to \cite{Browder, Skrypnik}.

\begin{definition}  Let $D \subset \wp$ be a set and let $F: D \rightarrow (\wp)'$ be a map. \begin{enumerate}[(i)] 
   \item  \it{\textbf {Demicontinuous}}: $F$ is said to be demicontinuous on $D$, if for any sequence $(\phi_n) \subset D$ such           that $\phi_n\rightarrow \phi_0$, then 
        $\displaystyle \lim_{n\rightarrow \infty} \left< F(\phi_n), \; \upsilon \right> = \left< F(\phi_0), \upsilon \right>,\; \forall \upsilon \in \wp.$ 
   \item \it{ \textbf{Class  $\al(D)$}}: $F$ is said to be in class $\al(D),$ if every sequence $(\phi_n)$ in $D$  satisfying  $\phi_n \rightharpoonup \phi_0$ and 
  $
\uplim_{n \rightarrow \infty} \left< F(\phi_n), \phi_n - \phi_0 \right> \leq 0,$ converges  to some $\phi_0$ in $D$. 
    \item  For $F \subset  \overline{D},$  $A(D,F)$ denotes the set of all bounded, demicontinuous map defined on   $\overline{D}$ that satisfies the class $\al(F).$ 
    \item  \it{ \textbf{Isolated zero}}: A point $\phi_0 \in D$ is called an isolated zero of $F$, if $F(\phi_0) = 0$ and there exists $r > 0$ such that the ball $B_{r}(\phi_0) $ (where $ \overline{B_{r}(\phi_0}) \subset D$) does not contain any other zeros of $F$.
    \item \it{\textbf{Degree}}: Let $F \in A(D,\pa D)$ satisfying $F(\phi) \neq 0$ for every $\phi \in \pa D.$ Let $(\upsilon_i)$ be a Schauder basis for $\wp$ and let $V_n = span\{\upsilon_1,..., \upsilon_n\}.$ A finite-dimensional approximation $F_n$ of $F$ with respect to $V_n$ is defined as: $$ F_n(\phi) = \sum_i^n \left< F(\phi),\upsilon_i \right> \upsilon_i, \; \text{for} \; \phi \in \overline{D_n}, \; \text{where} \; D_n = D \cap V_n.$$
    From \cite[Theorem 2.1]{Skrypnik}, $F_n(\phi) \neq 0$ for every $\phi \in \pa D_n$, the degree $deg(F_n, \overline{D_n}, 0)$ of $F_n$ with respect to  $0 \in V_n$ is well defined and independent of $n$. Further from \cite[Theorem 2.2]{Skrypnik}, $\lim_{n\rightarrow \infty} deg(F_n, \overline{D_n},0)$ is independent of basis $(v_i)$. Now the degree of $F$ with respect to $0 \in (\wp)'$ is defined as $$ deg(F, \overline{D}, 0) = \lim_{n\rightarrow \infty} deg(F_n, \overline{D_n},0).$$ 
    \item \it{\textbf{Homotopy}}: Let $F, G \in  A(D,\pa D)$ satisfying $F(\phi), G(\phi) \neq 0$ for every $\phi \in \pa D.$ The mapping $F$ and $G$ is said to be homotopic on $\overline{D}$, if there exists a sequence of one parameter family  $H_t: \overline{D} \rightarrow (\wp)'$, $t \in [0,1]$ such that $H_0 = F$ and $H_1 = G$ and $H_t$ satisfies the following:
    \begin{enumerate}[(a)]
        \item For $t \in [0,1]$, $H_t \in A(D, \pa D)$ and $H_t(\phi) \neq 0$ for  every $\phi \in \pa D.$ 
        \item For a sequence $t_n \in [0,1]$ satisfying $t_n \rightarrow t$ and  a sequence $\phi_n \in \overline{D}$ satisfying $\phi_n \rightarrow \phi_0,$  $H_{t_n}\phi_n \rightharpoonup H_{t} \phi_0$ as $n \rightarrow \infty$. \end{enumerate}
    \item \it{\textbf{Index}}: Let $F \in A(D, \overline{D})$ and let $\phi_0$ be an isolated zero of $F$. Then the index of a map $F$ is defined as  $ ind(F, \phi_0) =  \displaystyle\lim_{r\rightarrow 0} deg(F, \overline{B_{r}(\phi_0)}, 0).$ 
    \item  \it{\textbf{Potential operator}}: A map $F \in A(D,(\wp)')$ is called a potential operator, if there exists a functional $f: \wp \rightarrow \R$ such that $f^{\prime}(\phi) = F(\phi), \; $ for all $\phi \in \wp.$
\end{enumerate}
\end{definition}

The following Proposition is proved in \cite{Skrypnik} (Theorem 4.1, Theorem 4.4, Theorem 5.1, and Theorem 6.1).

\begin{proposition}\label{degree}  $(i)$ Let $F, G \in A(D, \pa D)$ satisfying $F(\phi), G(\phi) \neq 0$ for every $\phi \in \pa D.$ If $F$ and $G$ are homotopic in $\overline{D},$ then 
$deg(H_t, \overline{D}, 0) = C, \; \forall t \in [0,1].$ In particular, $deg(F, \overline{D}, 0) = deg(G, \overline{D}, 0)$.\vspace{0.3 cm} \\
$(ii)$ Let $F \in A(D, \pa D).$ Suppose that $0 \in \overline{D} \setminus \pa D$ and $\displaystyle \left< F(\phi), \phi \right> \geq 0,\; F(\phi) \neq 0$ for $\phi \in \pa D.$ Then $deg(F, \overline{D},0) = 1.$   \vspace{0.3 cm} \\
\noi $(iii)$ Let $F \in A(D, \overline{D})$ satisfying $F(\phi) \neq 0$, for every $\phi \in \pa D.$ If $F$ has only finite number of isolated zeros in $\overline{D},$ then $$ deg(F, \overline{D}, 0) = \sum_{i = 1}^n ind(F, \phi_i),$$ where $\phi_i (i = 1,...,n)$ are all zeros of $F$ in $D$.  \vspace{0.3 cm} \\
\noi $(iv)$ Let $F \in A(D,(\wp)')$ be a potential operator. Suppose that the point $\phi_0$  is a local minimum of $f$ and it is an isolated zero of $F$. Then $ind(F, \phi_0) = 1.$

\end{proposition}

%% file: Functional.tex
\section{Functional framework}

In this section, we set up a suitable functional framework for our problem. We consider the following functional on $\wp$:
$$  
G(\phi)=\dispa g |\phi|^p,  \quad \forall \phi \in \wp.
$$
For $g \in L^{\frac{N-1}{p-1}, \infty}(\pa \Om)$ (if $N > p$) and $g \in L^{1, \infty;N}(\pa \Om)$ (if $N=p$), Proposition \ref{Hardy boundary} ensures that $G$ is well defined. Now we study the continuity, compactness and differentiability of $G$.

\begin{proposition}\label{G cont}
Let $$ g \in \left\{\begin{array}{ll}
             L^{\frac{N-1}{p-1}, \infty}(\pa \Om) & \text{  for } N > p,\\
             L^{1, \infty; N}(\pa \Om) & \text{  for } N = p.
            \end{array}\right.$$ 
Then $G$ is continuous.
\end{proposition}
\begin{proof}
We only consider the case $N > p.$ For $N=p,$ the proof will follow using similar arguments.  Let $\phi_n \rightarrow \phi$ in $\wp$ and let $\ep >0$ be given. Clearly,
\begin{align*}
    \abs{G(\phi_n) - G(\phi)} \leq \dispa \abs{g}\abs{(\abs{\phi_n}^p - \abs{\phi}^p)}.
\end{align*}
 Using the inequality due to Lieb and Loss \cite[Page 22]{Lieb}, there exists $C = C(\ep, p) > 0$ such that
\begin{align*}
    \abs{(\abs{\phi_n}^p - \abs{\phi}^p)} \leq \ep \abs{\phi}^p + C \abs{\phi_n - \phi}^p \quad \text { a.e. on }  \pa \Om.
\end{align*}
Hence 
\begin{align}\label{cont1}
    \dispa \abs{g}\abs{(\abs{\phi_n}^p - \abs{\phi}^p)} \leq \ep\dispa \abs{g} \abs{\phi}^p + C \dispa \abs{g} \abs{\phi_n - \phi}^p.
\end{align}
Now using \eqref{N>p}, we obtain  
\begin{align}\label{cont2}
    \dispa \abs{g} \abs{\phi_n - \phi}^p \leq C \norm{g}_{\left( \frac{N-1}{p-1}, \infty \right)} \norm{\phi_n - \phi}^p_{\wp},
\end{align}
where $C = C(N, p) > 0$ is the embedding constant and $\p$ is the conjugate exponent of $p$. Now from \eqref{cont1} and \eqref{cont2}, we easily conclude  that $G(\phi_n) \rightarrow G(\phi)$ as $n\ra \infty$. 
\end{proof}

\begin{proposition}\label{G cpt}
Let $$g \in \left\{\begin{array}{ll}
                 \F_{\frac{N-1}{p-1}} & \text{  for } N > p,\\
                  \G_1 & \text{  for } N = p.
               \end{array}\right.$$  Then $G$ is compact.
\end{proposition}

\begin{proof} 
As before, we  only consider the case $N > p.$ Let $\phi_n \rightharpoonup \phi$ in $\wp$ and let  $\ep > 0$ be given. Set $L = \sup \{ \norm{ \phi_n}^p_{\wp} + \norm{\phi}^p_{\wp} \}. $ For $g \in \F_{\frac{N-1}{p-1}}$, we split $g = g_{\ep} + (g - g_{\ep})$ where $g_{\ep} \in \c(\pa \Om)$ such that $\norm{g - g_{\ep}}_{\left(\frac{N-1}{p-1}, \infty \right)} < \frac{\ep}{L}.$ Then 
\begin{align}\label{cpt1}
    \dispa \abs{g}\abs{(\abs{\phi_n}^p - \abs{\phi}^p)}  \leq \dispa \abs{g_{\ep}} \abs{(\abs{\phi_n}^p - \abs{\phi}^p)} + \dispa \abs{g - g_{\ep}} \abs{(\abs{\phi_n}^p - \abs{\phi}^p)}.
\end{align}
We estimate the second integral  of \eqref{cpt1} using \eqref{N>p} as, 
\begin{align}\label{cpt2}
    \dispa \abs{g - g_{\ep}} \abs{(\abs{\phi_n}^p - \abs{\phi}^p)} \leq C \norm{g - g_{\ep}}_{\left(\frac{N-1}{p-1}, \infty \right)}\left( \norm{\phi_n}^p_{\wp} + \norm{\phi}^p_{\wp}\right).
\end{align}
Since $\wp$ is compactly embedded into $L^p(\pa \Om)$ (Proposition \ref{classical}), there exists $n_1 \in \mathbb{N}$ such that $\int_{\pa \Om} \abs{g_{\ep}}\abs{(\abs{\phi_n}^p - \abs{\phi}^p)} < \ep, \; \forall n \geq n_1. $ Now from \eqref{cpt1} and \eqref{cpt2}, we obtain
$$ 
\dispa \abs{g}\abs{(\abs{\phi_n}^p - \abs{\phi}^p)} < (C+1) \ep, \quad \forall n \geq n_1. $$
Thus $G(\phi_n)$ converges to $G(\phi)$ as $ n\ra \infty.$ 
\end{proof}

\begin{proposition}\label{G' cpt}
Let $p \in (1, \infty).$ Let $N,g $ be given as in Proposition \ref{G cpt}. Then  $G$ is differentiable at every $\phi \in \wp$ and 
$$ 
\quad \big< G'(\phi), v \big> = p\dispa g| \phi|^{p-2} \phi v, \quad \forall v \in \wp.
$$
Moreover, the map $G'$ is compact.
\end{proposition}

\begin{proof}
For $\phi,v \in \wp$, let  $f : \pa \Om \times [-1,1] \rightarrow \R$ defined by $f(y , t) = g(y) \abs{(\phi + t v)(y)}^p.$
Then
$
\frac{\pa f}{\pa t}(\cdot,t) = p g \abs{\phi + t v}^{p-2} (\phi + t v) v
$ and
$$\left|\frac{\pa f}{\pa t}(\cdot,t)\right|\le p 2^{p-1} \abs{g} \left( \abs{\phi}^{p-1} +  \abs{v}^{p-1} \right) \abs{v}.$$
 Set $h = p 2^{p-1} \abs{g} \left( \abs{\phi}^{p-1} +  \abs{v}^{p-1} \right) \abs{v}$ and for each $n\in \N,$ set 
\begin{align*}
h_n(y) =  n \left( f(y, \frac{1}{n}) - f(y, 0) \right).
\end{align*}
  Clearly,  $h_n(y) \rightarrow \frac{\pa f}{\pa t}(y,0)$  a.e. on $\pa \Om$ and by mean value theorem, we also have
\begin{align*}
\abs{h_n(y)} \leq \sup_{t \in [-1,1]} \Big| \frac{\pa f}{\pa t}(y,t) \Big| \leq h(y).
\end{align*}
Furthermore, using a similar set of arguments as given in the proof of Proposition \ref{Hardy boundary}, one can show that  $h_n,h  \in L^1(\pa \Om),$ for each $n\in \N.$  Therefore, by the dominated convergence theorem, 
\begin{align*}
\lim_{n \rightarrow \infty}  \dispa  n \left( f(y, \frac{1}{n}) - f(y, 0) \right) \; \dy = \dispa  \frac{\pa f}{\pa t}(y,0) \; \dy= p\dispa g| \phi|^{p-2} \phi v.  
\end{align*}
Thus 
\begin{align*}
\big< G'(\phi), v \big> = \frac{{\rm d}}{\dt} G(\phi + tv) \Big|_{t = 0} = p\dispa g| \phi|^{p-2} \phi v.
\end{align*}
The proof of compactness is quite similar to that of Proposition \ref{G cpt}.
\end{proof}

For $p \in (1, \infty)$, consider the following functional
$$
J(\phi) = \dis \abs{\Gr \phi}^p, \quad \forall \phi \in \wp.
$$ 
Then $J$ is  differentiable on $\wp$, and the derivative is given by
$$ 
\big< J'(\phi), u \big> = p\dis |\Gr \phi|^{p-2}\Gr \phi\cdot \Gr u, \quad \forall u \in \wp.
$$

\begin{proposition}\label{class alpha}
Let $p \in (1, \infty).$ Then 
\begin{enumerate}[(i)]
    \item $J'$ is continuous.
    \item $J'$ is of class $\al(\wp).$
\end{enumerate}
\end{proposition}

\begin{proof}
(i)  Let $\phi_n \rightarrow \phi$ in $\wp$. For $v \in \wp,$ 
\begin{align*}
    \big|\left< J'(\phi_n) - J'(\phi), v \right>\big| & \leq \dis \abs{(\abs{\Gr \phi_n}^{p-2} \Gr \phi_n - \abs{\Gr \phi}^{p-2} \Gr \phi)} \abs{\Gr v} \\
    & \leq \left( \dis \abs{(\abs{\Gr \phi_n}^{p-2} \Gr \phi_n - \abs{\Gr \phi}^{p-2} \Gr \phi)}^{\p} \right)^{\frac{1}{\p}} \left( \dis \abs{\Gr v}^p \right)^{\frac{1}{p}}. 
\end{align*}
Therefore,
\begin{align*}
    \norm{J'(\phi_n) - J'(\phi)} \leq \left( \dis \abs{(\abs{\Gr \phi_n}^{p-2} \Gr \phi_n - \abs{\Gr \phi}^{p-2} \Gr \phi)}^{\p} \right)^{\frac{1}{\p}}.
\end{align*}
Now consider the map $J_1$ defined as $J_1(\phi) = \abs{\Gr \phi}^{p-2} \Gr \phi.$ Clearly $J_1$ maps $\wp$ into $L^{\p}(\Om)$ and $J_1$ is continuous. Hence we conclude $ \norm{J'(\phi_n) - J'(\phi)} \rightarrow 0$ as $n \rightarrow \infty.$

\noi (ii) Let  $\phi_n \rightharpoonup \phi$ in $\wp$ and let $\uplim_{n \rightarrow \infty} \big< J'(\phi_n), \phi_n - \phi \big> \leq 0.$ Then 
\begin{align}\label{ca 1}
     \uplim_{n \rightarrow \infty}  \big< J'(\phi_n) - J'(\phi), \phi_n - \phi \big> = \uplim_{n \rightarrow \infty}  \big< J'(\phi_n), \phi_n - \phi \big> - \lowlim_{n \rightarrow \infty}  \big< J'(\phi), \phi_n - \phi \big> \leq 0.
\end{align}
Now for each $n \in \N,$ 
\begin{align*}
    \big< J^{\prime}(\phi_n) -  J^{\prime}(\phi), \phi_n - \phi \big>  \geq p\left( \Vert \Gr \phi_n \Vert_p^{p-1} - \Vert \Gr \phi \Vert_p^{p-1} \right) \left( \Vert \Gr \phi_n \Vert_p - \Vert \Gr \phi \Vert_p \right) \geq 0.
\end{align*}
Hence from \eqref{ca 1}, we get 
\begin{align*}
    \lim_{n \rightarrow \infty} \big< J'(\phi_n) - J'(\phi), \phi_n - \phi \big> = 0.
\end{align*}
Therefore, $\Vert \Gr \phi_n \Vert_p \rightarrow \Vert \Gr \phi \Vert_p$ as $n \rightarrow \infty$. Hence by uniform convexity of $(L^p(\Om))^N$, we obtain $\Gr \phi_n \rightarrow \Gr \phi$ in $(L^p(\Om))^N$.  Further, since  $\wp$ is compactly embedded into $L^p(\Om)$, we get $\phi_n \rightarrow \phi$ in $L^p(\Om)$ . Therefore, $\phi_n\rightarrow \phi $ in $\wp.$ Thus the map $J'$ is of class $\al(\wp).$
\end{proof}

\begin{proposition}\label{compactmap1}
Let $p \in (1, \infty)$ and let $N$, $r$ and $f$ satisfy  \textbf{(H1)} or \textbf{(H2)}. Then the map $F$ defined by 
$$ 
\big< F(\phi), v \big> =  \dispa f r(\phi) v 
$$
is a well-defined map from $\wp \rightarrow (\wp)^{\prime}$. Moreover, $F$ is  continuous and compact.
\end{proposition}

\begin{proof} 
First, we assume that $N, r$ and $f$ satisfy \textbf{(H1)}. In this case $\ga\in (1,\pcr)$ and we use different arguments for   $\ga \in (1,p)$ and $\ga \in [p, \pcr).$
For $\ga \in (1,p),$  there exists $C > 0$ such that $\abs{r(s)} \leq C \abs{s}^{p-1}$ for $s \in \R$. Therefore, using the finer trace embeddings (Proposition \ref{Cianchi}), for $\phi, v \in \wp,$ clearly we have 
\begin{align}\label{for < p}
    \big |{\big< F(\phi), v \big>} \big | \leq C \norm{f}_{\left(\frac{N-1}{p-1}, \infty \right)}\norm{\phi}^{p -1}_{\wp} \norm{v}_{\wp}.
    \end{align} For $\gamma \in [p, \pcr)$,  using Proposition \ref{Lorentz properties} and the finer trace embeddings (Proposition \ref{Cianchi}), we have
    \begin{align}\label{trace embed 1}
    \wp  \hookrightarrow L^{\pcr, \ga}(\pa \Om).
\end{align}
Since 
$
\frac{1}{\tilde{p}} + \frac{(\gamma-1)(N-p)}{p(N-1)} + \frac{N-p}{p(N-1)} = 1,
$
 for $\phi, v \in \wp,$  using the generalized H\"{o}lder inequality (Proposition \ref{Lorentz properties}), we obtain 
\begin{align*}
    \dispa \abs{f} \abs{r(\phi)v}  \leq C \tilde{p} \norm{f}_{(\tilde{p}, \infty)} \norm{\phi}^{\ga - 1}_{\left( \pcr, \ga \right)} \norm{ v}_{\left(\pcr, \ga \right)}.
\end{align*}
Therefore, from \eqref{trace embed 1},   
\begin{align}\label{cont}
    \big |{\big< F(\phi), v \big>} \big | \leq C \norm{f}_{(\tilde{p}, \infty)} \norm{\phi}^{\ga -1}_{\wp} \norm{v}_{\wp},  \quad \forall \phi, v \in \wp,
\end{align}
where $C = C(N,p) > 0$. \\ 
Now assume that $N,r$ and $f$ satisfy \textbf{(H2)}. For $d \in (1, \infty)$, choose $a_i, b_i  \in (1, \infty)$ (for $i = 1,2$) such that 
\begin{align*}
    a_1, b_1 > \frac{1}{\ga - 1}, \quad   \frac{1}{d} + \frac{1}{a_1} + \frac{1}{a_2} = 1 = \frac{1}{N} + \frac{1}{b_1} + \frac{1}{b_2}.
\end{align*}
For $\phi, v \in \wp,$ using the generalized H\"{o}lder inequality (Proposition \ref{Lorentz properties}), we obtain 
\begin{align}\label{cont 1}
    \dispa \abs{f} \abs{r(\phi) v} &\leq  C d \norm{f}_{(d, N)} \norm{\phi}^{\ga - 1}_{(a_1(\ga - 1),b_1(\ga - 1))} \norm{v}_{(a_2, b_2)}.
\end{align}
Now by Proposition \ref{properties} and using the trace embeddings (Proposition \ref{classical} and Proposition \ref{Cianchi}), we have 
\begin{equation*}
    \begin{aligned}
& \quad  L^{d, \infty;N}(\pa \Om) \hookrightarrow L^{d,N}(\pa \Om),\\
& \quad \wN  \hookrightarrow L^{\infty, N; -1}(\pa \Om) \hookrightarrow L^{a_1(\gamma-1),b_1(\gamma-1)}(\pa \Om),\\
& \quad \wN \hookrightarrow L^q(\pa \Om) \hookrightarrow  L^{a_2, b_2}(\pa \Om), \; \text{for} \; q > a_2.
    \end{aligned} 
\end{equation*}
Therefore, from \eqref{cont 1} we get \label{cont 2}
\begin{align*}
    \big |{\big< F(\phi), v \big>} \big | \leq C\norm{f}_{(d, \infty;N)}\norm{\phi}^{\gamma-1}_{\wN} \norm{v}_{\wN}, \quad \forall \phi, v \in \wN,
\end{align*}
where $C=C(N) > 0.$ Thus the map $F$ is well defined in both the cases.  The continuity and the compactness of  $F$ will follow from the similar set of arguments as given  in the proof of Proposition \ref{G cpt}. So we omit the proof.
\end{proof}

\begin{proposition}\label{Growth}
Let $p \in (1, \infty)$. Let $N,r$ and $f$ be given as in Proposition  \ref{compactmap1}. Then 
$$
 \frac{\norm{F(\phi)}_{(\wp)'}}{\norm{\phi}_{\wp}^{p-1}} \longrightarrow 0, \quad \text{as} \;  \norm{\phi}_{\wp} \rightarrow 0.
$$
\end{proposition}

\begin{proof}
Let $\ep > 0$ be given. We only prove the case when $N,r$ and $f$ satisfy \textbf{(H1)}. For  \textbf{(H2)}, the proof is similar. For $\ga \in [p, \pcr),$  using \eqref{cont} we have, 
\begin{align*}
      \norm{F(\phi)} \leq C \norm{f}_{(\tilde{p}, \infty)}\norm{\phi}^{\ga -1}_{\wp}, \quad \forall \phi \in \wp.
\end{align*}
 Therefore, 
\begin{align*}
      \frac{\norm{F(\phi)}_{(\wp)'}}{\norm{\phi}_{\wp}^{p-1}} \leq C \norm{f}_{(\tilde{p}, \infty)}\norm{\phi}^{\ga -p}_{\wp}.
\end{align*}
If  $\ga \in (1,p)$,  then from \textbf{(H1)} there exists $s_0 > 0$ and $C = C(s_0) >  0$ such that
\begin{equation}\label{growth}
\begin{aligned}
       \abs{r(s)} &< \frac{\ep}{\norm{f}_{\left(\frac{N-1}{p-1}, \infty \right)}}  \abs{s}^{p-1}, \quad \text{for} \; \abs{s} < s_0, \\
       \abs{r(s)} \leq C \abs{s}^{p -1} \quad &\text{and} \quad \abs{r(s)} \leq C \abs{s}^{\pcr-1}, \quad \text{for} \; \abs{s} \geq s_0. 
\end{aligned}
\end{equation}
For $\phi \in \wp,$ set $A = \{ y \in \pa \Om : \abs{ \phi(y)} < s_0 \}$ and $B = \pa \Om \setminus A.$ For $v \in \wp$, using \eqref{growth} and \eqref{for < p}, we get
\begin{align}\label{grth1}
       \int_{A} \abs{f}\abs{r(\phi)}\abs{ v} < \frac{\ep}{\norm{f}_{\left(\frac{N-1}{p-1}, \infty \right)}} \int_{A} \abs{f} \abs{\phi}^{p-1} \abs{v} & \leq C \ep \norm{\phi}_{\wp}^{p-1} \norm{v}_{\wp}.
\end{align}
To estimate the above integral on $B$, we split $f= f_\ep+ (f-f_\ep)$ where  $f_{\ep} \in \c(\pa \Om)$ with $\norm{f - f_{\ep}}_{\left(\frac{N-1}{p-1}, \infty \right)} < \ep.$ Now \eqref{growth} and \eqref{for < p} yield
\begin{align}\label{grth2}
        \int_{B} \abs{f - f_{\ep}} \abs{r(\phi)} \abs{v} \leq C\int_{B} \abs{f - f_{\ep}} \abs{\phi}^{p-1} \abs{v}  < C\ep \norm{\phi}_{\wp}^{p-1}\norm{v}_{\wp},
\end{align}
where $C=C(s_0,N,p)>0.$ On the other hand using \eqref{growth}, H\"{o}lder inequality (Proposition \ref{properties}) and the classical trace embeddings (Proposition \ref{classical}), we obtain
\begin{align*}
\nonumber  \int_{B} \abs{f_{\ep}}\abs{r(\phi)}\abs{v}  & \leq C \int_{B}  \abs{f_{\ep}}\abs{\phi}^{\pcr - 1} \abs{v} \\ 
            & \leq C \norm{f_{\ep}}_{L^{\infty}(\pa \Om)} \norm{\abs{\phi}^{\pcr - 1}}_{L^{\frac{p(N-1)}{N(p-1)}}(\pa \Om)} \norm{v}_{L^{\frac{p(N-1)}{N-p}}(\pa \Om)},\\
            &\leq C \norm{f_{\ep}}_{L^{\infty}(\pa \Om)} \norm{\phi}_{\wp}^{\frac{N(p-1)}{N-p}} \norm{v}_{\wp},  
\end{align*}
where $C = C(N,p) > 0.$  Now using \eqref{grth2} we conclude 
\begin{align*} 
           \int_{B} \abs{f}\abs{r(\phi)}\abs{v} \leq C \left( \ep \norm{\phi}_{\wp}^{p-1}  +  \norm{f_{\ep}}_{L^{\infty}(\pa \Om)} \norm{\phi}_{\wp}^{\frac{N(p-1)}{N-p}} \right) \norm{v}_{\wp},
\end{align*}
where $C = C(s_0,N,p) > 0$. Thus  \eqref{grth1} and the above inequality yield:   
\begin{align*}
           \norm{F(\phi)}_{(\wp)'} <C \left( \ep \norm{\phi}_{\wp}^{p-1}  +  \norm{f_{\ep}}_{L^{\infty}(\pa \Om)} \norm{\phi}_{\wp}^{\frac{N(p-1)}{N-p}} \right).
\end{align*}
 Therefore, 
 \begin{align*}
           \frac{\norm{F(\phi)}_{(\wp)'}}{\norm{\phi}_{\wp}^{p-1}} < C \left( \ep  +  \norm{f_{\ep}}_{L^{\infty}(\pa \Om)} \norm{\phi}_{\wp}^{\frac{p(p-1)}{N-p}} \right)\ra 0.
\end{align*}
as $\norm{\phi}_{\wp} \rightarrow 0.$ 
\end{proof}

For $g$ as given in Theorem \ref{Steklov existence}, we consider the set 
$$ 
M_g = \left\{ \phi \in \wp: \dispa g \abs{\phi}^p > 0 \right\}.
$$ 
Since $g^+ \not \equiv 0$, we can show that the set $M_g$ is nonempty. The functional $J$ is not coercive on $\wp.$ However, using a Poincar\'{e} type inequality on $M_g$ we show that $J$ is coercive on $M_g$.

\begin{Lemma}\label{Poin}
Let  $g^+ \not \equiv 0,$  $\int_{\pa \Om} g < 0,$ and 
$$
g \in \left\{\begin{array}{ll}                                                     
                \F_{\frac{N-1}{p-1}} & \text{  for } N > p,\\
                 \G_1 & \text{  for } N = p. 
              \end{array}\right.$$ 
Then there exists $m \in (0,1)$ such that
\begin{align}\label{Poincare}
      \dis \abs{\Gr \phi}^p \geq m \dis \abs{\phi}^p, \quad \forall \phi \in M_g.
\end{align}
\end{Lemma}
\begin{proof}
 On the contrary, assume that \eqref{Poincare} does not hold for any $m\in (0,1)$. Thus for each $n\in \N,$  there exists $\phi_n \in M_g$ such that
 \begin{align*}
    \dis \abs{\Gr \phi_n}^p < \frac{1}{n} \dis \abs{\phi_n}^p.
 \end{align*}
 If we set $w_n = \norm{\phi_n}^{-1}_p \phi_n,$ then  $\norm{w_n}_p = 1$ and $ \int_{\Om} \abs{\Gr w_n}^p < \frac{1}{n}.$
 Thus $(w_n)$ is bounded and hence there exists a subsequence $(w_{n_k})$ of $(w_n)$ such that  $w_{n_k} \rightharpoonup w$ in $\wp.$ By weak lowersemicontinuity of $\norm{\Gr \cdot}_p$ we have  $\norm{\Gr w}_p =0.$ Hence  the connectedness yields $w \equiv c $ a.e. in $\overline{\Om}$. By the compactness of the embedding of $\wp$ into $L^p(\Om)$, we get $\norm{w}_p = 1$ and hence $\abs{c} |\Om|^{\frac{1}{p}} = 1$. Therefore, $ \int_{\pa \Om} g \abs{w}^p = \frac{1}{|\Om|} \int_{\pa \Om} g < 0$. On the other hand, $\int_{\pa \Om} g \abs{w_{n_k}}^p = \norm{\phi_{n_k}}^{-p}_p \int_{\pa \Om} g \abs{\phi_{n_k}}^p > 0 $. Thus by the compactness of $G$ (Proposition \ref{G cpt}), we get  $ \int_{\pa \Om} g \abs{w}^p = \lim_{k \rightarrow \infty} \int_{\pa \Om} g \abs{w_{n_k}}^p \geq 0$,   a contradiction. Thus there must exists $m\in (0,1)$ satisfying  \eqref{Poincare}. 
\end{proof}

\begin{remark}\label{Manifold}
For $g$ as given in Lemma \ref{Poin}, consider the set
\begin{align*}
      N_g = \left \{\phi \in \wp: \dispa g \abs{\phi}^p = 1\right \} = G^-(1).
\end{align*}  
For $\phi \in N_g,$  $ \langle  G'(\phi), \phi \rangle = p \neq 0$. Thus $1$ is a regular point of $G$ and $N_g$  is a $C^1$ manifold. Moreover (see \cite[Proposition 6.4.35]{Drabek-Milota}),
\begin{align*}
      \norm{\dJ(\phi)} = \min_{\la \in \R} \norm{(J' - \la G')(\phi)}, \quad \forall \phi \in N_g. 
\end{align*}
\end{remark}

\begin{definition}
A map $f \in C^1(Y, \R)$ is said to satisfy {\bf Palais-Smale (P. S.)} condition on a $C^1$ manifold $ M  \subset Y$, if $(\phi_n)$ is a sequence  in $M$ such that $f(\phi_n) \rightarrow c \in \R$ and $ \Vert {\rm d}f (\phi_n) \Vert \rightarrow 0,$ then $(\phi_n)$ has a subsequence that converges in $M$.
\end{definition}

\begin{Lemma}\label{PS} 
Let $g$ be  as given in Lemma \ref{Poin}. Then $J$ satisfies the P. S. condition on $N_g.$ 
\end{Lemma}

\begin{proof}
Let $(\phi_n)$ be a sequence in $N_g$ and $\la \in \R$ such that $J(\phi_n) \rightarrow \la$ and $\norm{\dJ(\phi_n)} \rightarrow 0.$ By Remark \ref{Manifold}, there exists a sequence $(\la_n)$ such that $(J' - \la_n G')(\phi_n) \rightarrow 0$ as $n \rightarrow \infty.$ By Lemma \ref{Poin}, the sequence $(\phi_n)$ is also bounded in $\wp.$ Now using the reflexivity of $\wp$, we get a subsequence $(\phi_{n_k})$ such that $\phi_{n_k} \rightharpoonup \phi$ in $\wp$. Since $N_g$ is weakly closed, $\phi \in N_g$. Also $\la_{n_k} \rightarrow \la$ as $k \rightarrow \infty,$ since $$\big< (J' - \la_{n_k} G')(\phi_{n_k}), \phi_{n_k} \big> = p(J(\phi_{n_k}) - \la_{n_k}).$$ Furthermore,
\begin{align*}
       \big< J'(\phi_{n_k}), \phi_{n_k} - \phi \big> = \big< (J' - \la_{n_k} G')(\phi_{n_k}), \phi_{n_k} - \phi \big> + \la_{n_k} \big< G'(\phi_{n_k}), \phi_{n_k} - \phi \big>.
\end{align*}
Now using the compactness of $G'$, we get  $\langle J'(\phi_{n_k}), \phi_{n_k} - \phi \rangle \rightarrow 0.$  Moreover, as $J'$ is of class $\al(\wp)$ (Proposition \ref{class alpha}), the sequence $(\phi_{n_k})$ converges to $\phi$ in $\wp$. Therefore, $J$ satisfies the P. S.  condition on $N_g.$ 
\end{proof}

%% file: Proofs.tex
\section{Proof of main theorems}

In this section, we prove all our main theorems. 
\subsection{The existence  and some of the properties of the first eigenvalue}
\noi{\textbf{Proof of Theorem \ref{Steklov existence}}}: \\
First, recall that 
$$
\la_1 =  \inf_{\phi \in N_g} \dis \abs{\Gr \phi}^p. 
$$ 
From Lemma \ref{Poin}, we clearly have $\la_1 > 0$. Since the functional $J$ is coercive on $N_g $, a sequence that minimizes $J$ over $N_g$  will be bounded and hence admits a weakly convergent subsequence that converges to say  $\phi_1$. As $N_g$ is weakly closed, $\phi_1\in N_g$ and $J(\phi_1)=\la_1.$ Thus $\la_1$ is the minimum of $J$ on $N_g$ and hence $\norm{\dJ(\phi_1)} = 0.$ Now from  Remark \ref{Manifold}, we obtain 
\begin{align}\label{weak form of first}
     \dis |\Gr \phi_1|^{p-2} \Gr \phi_1 \cdot \Gr v\; \dx = \la_1 \dispa  g \abs{\phi_1}^{p-2}\phi_1 v \; \dsg, \quad \forall v \in \wp.
\end{align}

\noi{\it $\la_1$ is a principal eigenvalue}: Clearly $\abs{\phi_1}$ is also an eigenfunction of \eqref{Steklov weight} corresponding to $\la_1$. Moreover, as $\abs{\phi_1}$ is $p$-harmonic, $\abs{\phi_1} \in C^{1, \alpha}(\Om)$. Since $\abs{\phi_1} \geq 0$, by the maximum principle in \cite[Theorem 5]{Vazquez}, $\abs{\phi_1} > 0$ in $\Om$. Without loss of generality we may  assume $\phi_1>0$ in $\Om.$ We show that $\phi_1$ is positive also on $\partial \Om.$
For $\ep > 0$, consider the function $\frac{\phi_1}{\phi_1+ \ep}$. It is easy to verify that  $\frac{\phi_1}{\phi_1+ \ep} \in \wp$ and $\frac{\phi_1}{\phi_1+ \ep}  \rightarrow 1$ in $L^p(\Om).$ We show that $\frac{\phi_1}{\phi_1+ \ep}  \rightarrow 1$ in $\wp$ as well. This together with trace embedding will ensure that $\phi_1 > 0$ in $\overline{\Om}$.
Thus it is enough to prove  $\Gr \frac{\phi_1}{\phi_1+ \ep} \rightarrow 0$ in $L^p(\Om)$ as $\ep\ra 0.$ Notice that,
\begin{align}\label{bounded}
     \left|{\Gr \frac{\phi_1}{\phi_1 + \ep}}\right|^p = \left( \frac{\ep}{\phi_1 + \ep} \right)^p \frac{\abs{\Gr \phi_1}^p}{(\phi_1 + \ep)^p} \leq \frac{\abs{\Gr \phi_1}^p}{\phi_1^p}.
\end{align}
Furthermore, by taking $ \frac{1}{(\phi_1 + \ep)^{p-1}} \in \wp$ as a test function in \eqref{weak form of first}, we obtain
\begin{align*}
     (p-1)\dis \frac{\abs{\Gr \phi_1}^p}{(\phi_1 + \ep)^p} = \la_1 \dispa g  \left( \frac{\phi_1}{\phi_1 + \ep} \right)^{p-1} \leq \la_1 \dispa \abs{g}.   
\end{align*}
We apply Fatou's lemma and let $\ep\ra 0$ in the above inequality to get  
\begin{align*}
     (p-1)\dis \frac{\abs{\Gr \phi_1}^p}{\phi_1^p} \leq  \la_1 \dispa \abs{g}. 
\end{align*} 
Now \eqref{bounded} together with the dominated convergence theorem ensures that $\Gr \frac{\phi_1}{\phi_1+ \ep} \rightarrow 0$ in $L^p(\Om)$.

\noi{\it The uniqueness and the simplicity}: The usual arguments (for example, see  \cite[Lemma 3.1]{Torne} for a proof) using the Picone's identity  \cite[Theorem 1.1]{Picone} gives the  uniqueness of the positive principal eigenvalue  and the simplicity of $\la_1$.

\noi{\it $\la_1$ is an isolated eigenvalue}: We adapt the proof of \cite[Proposition 2.12]{ADSS}. On the contrary, we suppose that there exists a sequence $(\la_n)$ of eigenvalues of \eqref{Steklov weight} converging to $\la_1$. For each $n\in \N$, let  $\psi_n \in N_g$ be an eigenfunction corresponding to $\la_n$. Then $J(\psi_n) = \la_n \rightarrow \la_1$ and 
$$
\big< (J' - \la_n G')(\psi_n), \psi_n \big> = (J - \la_n G)(\psi_n) = 0, 
$$ 
i.e., $\norm{\dJ(\psi_n)} = 0$. Hence using Lemma \ref{PS} and the continuity of $J'$ and $G'$, we get $\psi_n \rightarrow \psi$, an  eigenfunction corresponding to $\la_1$. Since $\la_1$ is simple, $\psi = \pm \phi_1,$ where $\phi_1$ is a first eigenfunction such that $\phi_1>0$ on $\overline{\Om}.$ If we let  $\psi= {\phi_1}$, then by Egorov's theorem there exists $E \subset \Om$ and $n_1\in \N$ such that $\abs{E} < \ep$  and $\psi_n^{-}= 0$ a.e. in $E^c$ for  $n \geq n_1.$  Also from \eqref{Steklov weight} we have 
\begin{align*}
     \dis \abs{\Gr \psi_n^{-}}^p = \la_n \dispa g \abs{\psi_n^{-}}^p.
\end{align*}
 Notice that $\int_{\Om} \abs{\Gr \psi_n^{-}}^p\neq 0,$ since $\psi_n$ changes sign on $\Om.$ Now by setting $v_n= (\int_{\pa \Om} g \abs{\psi_n^{-}}^p)^{-\lfrac{1}{p}} \psi_n^{-},$ we have $v_n \in N_g$ and $\int_{\Om} \abs{\Gr v_n}^p = \la_{n} \rightarrow \la_1.$ Therefore,  $v_n$ must converge to ${\phi_1},$ a contradiction as  $v_{n} = 0$ a.e. in $E^c$ for $n\ge n_1$. Thus  $\la_1$ must be an  isolated eigenvalue.  \qed

\begin{remark}\label{Stricly contained}
\begin{enumerate}[(a)]
    \item Let $$g \in \left\{\begin{array}{ll}
                              L^{\frac{N-1}{p-1}, \infty}(\pa \Om) & \text{  for } N > p,\\
                              L^{1, \infty; N}(\pa \Om) & \text{  for } N = p.
                       \end{array}\right.$$ 
          Then $\frac{1}{\la_1}$ is the best constant in the following weighted trace inequality:
          \begin{align*}
            \dispa \abs{g} \abs{\phi}^p &\leq C  \dis \abs{\Gr \phi}^p, \quad \forall  \phi \in \wp. 
          \end{align*}
           In addition, if $g$ satisfy all the assumptions of  Theorem \ref{Steklov existence}, then this best constant is also attained.
    \item  Since $\Om$ is bounded, we have  
            $$ L^q(\pa \Om)\subset L^{\llfrac{N-1}{p-1}}(\pa \Om), \; \forall q > \frac{N-1}{p-1}, \: \text{ and } \: \;     L^q(\pa \Om) \subset \G_1,\; \forall q \in (1, \infty). 
            $$ 
            Thus,  Theorem 1.2 of  \cite{Torne} follows from Theorem \ref{Steklov existence}. Furthermore, Example \ref{Ex2}  and Example \ref{Ex3} give examples  of weight functions for which  Theorem 1.2 of \cite{Torne} is not applicable, however admits a positive principal eigenvalue by our Theorem \ref{Steklov existence}.
\end{enumerate}
\end{remark}

\begin{remark}
For $g$ as given in Theorem \ref{Steklov existence}, the functional $J$ and the set $N_g$ satisfy all the properties of  \cite[Theorem 5.3]{Kavian}. Therefore, by \cite[Theorem 5.3]{Kavian}, there exists a sequence of eigenvalues $(\la_n)$ of \eqref{Steklov weight} and the sequence $(\la_n)$ is unbounded.
\end{remark}

\subsection{Bifurcation}

For proving Theorem \ref{bifur}, we adapt the degree theory arguments given in \cite{Drabek-Huang}, also see \cite{Anoopthesis}. We split our proof into several lemmas and propositions. 
\begin{Lemma}\label{coercive}
Let  $g^+ \not \equiv 0,$  $\int_{\pa \Om} g < 0,$ and 
$$ g \in \left\{\begin{array}{ll}                                                                
                  \F_{\frac{N-1}{p-1}} & \text{  for } N > p,\\
                  \G_1 & \text{  for } N = p.
                \end{array} \right.
                $$ 
Let $(\phi_n)$ be a sequence in $\wp$ such that 
\begin{align}\label{ineq5}
\dis \abs{\Gr \phi_n}^p - \la \dispa g \abs{\phi_n}^p < C
\end{align} 
for some $C>0$ and $\la>0.$ If $(\norm{\Gr\phi_n}_p)$ is bounded, then $(\norm{\phi_n}_p)$ is bounded.
\end{Lemma}

\begin{proof}  Our proof is by method of contradiction. Suppose that the sequence $(\norm{\Gr\phi_n}_p)$ is bounded and $\norm{\phi_n}_p\ra \infty$ as $n\ra \infty.$  By setting $w_n = \norm{\phi_n}^{-1}_p \phi_n,$ we obtain  $\norm{w_n}_p = 1$ and $\norm{\Gr w_n}_p \rightarrow 0$ as $n \rightarrow \infty.$ Thus there exists a subsequence $(w_{n_k})$ of $(w_{n})$ such that $w_{n_k} \rightharpoonup w$ in $\wp$. Now the weak lowersemicontinuity of $\norm{\Gr \cdot}_p$ gives $\norm{\Gr w}_p=0$. Since $\Om$ is connected, we get $w=c$ a.e. in $\overline{\Om}$ and from the compactness of the embedding of $\wp$ into $L^p(\Om)$, $\abs{c} |\Om|^{\frac{1}{p}} = 1$. Thus $\int_{\pa \Om} g \abs{w}^p = \frac{1}{|\Om|} \int_{\pa \Om} g < 0$. On the other hand from  \eqref{ineq5} we also have
\begin{align*}
    \dis \abs{\Gr w_{n_k}}^p - \la \dispa g \abs{w_{n_k}}^p \leq \frac{C}{\norm{\phi_{n_k}}^p_p}. 
\end{align*}
Now we let $k \rightarrow \infty$ so that the compactness of $G$ gives $- \la \int_{\pa \Om} g \abs{w}^p \leq 0$.  A contradiction to $\int_{\pa \Om} g \abs{w}^p < 0$. 
\end{proof}

In the next proposition, for $\la\in (0,\la_1+\de),$ we find a lower estimate of the functional $J-\la G.$ 
\begin{proposition}\label{bdd below}
Let $ \de > 0$ and let $\la\in(0,\la_1+\de)\setminus{\la_1}.$ Then for $\phi\in \wp\setminus\{0\},$ 
\begin{align}\label{bounded below}
J(\phi) - \la G(\phi) >\left\{\begin{array}{ll}
                              0, & \text{ if } \la\in(0,\la_1); \\
                              \frac{-\de}{\la_1}J(\phi), & \text{ if } \la\in(\la_1,\la_1+\de). 
                              \end{array}
                        \right.
\end{align}
\end{proposition}

\begin{proof}
 Firstly, for any $\la>0$  and  $\phi\in \wp\setminus\{0\},$ we consider the following  cases:
\begin{enumerate}[(i)]
      \item $G(\phi) \le 0$ and $J(\phi) > 0:$  clearly $J(\phi) - \la G(\phi) > 0.$
      \item $G(\phi)= 0$ and $J(\phi)=0:$   using the connectedness of $\Om$ and the fact that $\int_{\pa \Om} g <0$, we get $\phi =0$. So this case does not    arise, since $\phi\neq 0.$
      \item $G(\phi)>0:$ in this case $\la_1 \leq \frac{J(\phi)}{G(\phi)}.$ Thus for $\la\in(0,\la_1),$ we get  $J(\phi) - \la G(\phi) > 0$.
\end{enumerate}
  Secondly, for $\la \in (\la_1, \la_1 +\de)$ and $\phi \in \wp,$ we have
\begin{align}\label{lb1}
    J(\phi) - \la G(\phi)&= J(\phi) - \la_1 G(\phi) + (\la_1 - \la) G(\phi) \no \\ &\geq (\la_1 - \la) G(\phi) >\frac{\la_1 - \la}{\la_1} J(\phi)>- \frac{\de}{\la_1} J(\phi),
\end{align}
where the inequalities follow from the facts $J(\phi) - \la_1 G(\phi) \geq 0$  and $\la\in (\la_1,\la_1+\de) $.
\end{proof}

For  $\la \in (\la_1,\la_1+\de)$, we consider a differentiable function $\eta(t)$ such that
\begin{align}\label{df}
\eta(t)
 =\left\{\begin{array}{ll}
                0, &\; 0\le t \leq 1, \\
                \text {strictly convex}, &\, 1<t<2,\\
                \frac{2 \de}{\la_1}(t - 1), &\; t \geq 2.
         \end{array} 
 \right. 
\end{align}
Therefore,
\begin{align}\label{df1}
   \eta^{\prime}(t)= \;\left\{\begin{array}{ll} 
                               0, &\; 0\le t <1; \\ 
                               \frac{2 \de}{\la_1}, &\; t \ge 2,
                            \end{array} 
                        \right. \: \text { and } \: \eta^{\prime}(t)\ge 0,\; 1\le t\le 2.  
\end{align}
A similar function (with an additional parameter $k$) is considered in the proof of \cite[Theorem 4.1]{Drabek-Huang}. We would like to point out that, their  proof also works by fixing a value for $k$. Since,  the functional $J-\la G$ is not bounded below for $\la \in (\la_1,\la_1+\de)$, we add a non-negative term to it. 
The following result is proved  as a part of the proof of \cite[Theorem 4.1]{Drabek-Huang}.

\begin{Lemma}\label{map properties}
 Let $\la \in (\la_1,\la_1+\de)$ and let $\eta$ be  given as above. Then the functional $\eta_{\la}(\phi) = J(\phi) - \la G(\phi) + \eta(J(\phi))$ satisfies the following: 
 \begin{enumerate}[(a)]
     \item $\eta_\la$ is weakly lower semicontinuous.
     \item $\eta_\la$ is coercive.
     \item $\eta_\la$  is bounded below.
     \item there exists $R_0>0$ such that the map $\eta_\la' : \wp\ra (\wp)'$ does not vanish on $\partial B_R(0)$ for all $R\ge R_0.$
 \end{enumerate}
\end{Lemma}
\begin{proof}
$(a)$  Let $\phi_n \rightharpoonup \phi$ in $\wp.$ Since $J$ is weakly lower semicontinuous, $G$ is compact and $\eta$ is increasing and continuous, we get
\begin{align*}
    \lowlim_{n \rightarrow \infty} \eta_{\la}(\phi_n) & = \lowlim_{n \rightarrow \infty} J(\phi_n) - \la \lowlim_{n \rightarrow \infty} G(\phi_n) +  \eta (\lowlim_{n \rightarrow \infty} (J(\phi_n))) \\
    & \geq J(\phi) - \la G(\phi) + \eta (J(\phi)) = \eta_{\la}(\phi).
\end{align*} 
Therefore, $\eta_\la$ is weakly lower semicontinuous. 

\noi $(b)$ Let $(\phi_n)$ be a sequence in $\wp$ such that $\eta_{\la}(\phi_n)\le C, \forall n\in \N.$ We show that the sequence  $(\phi_n)$ is bounded in $\wp.$ 
 From \eqref{bounded below}, we have
\begin{equation}\label{coercive 2}
    C\ge \eta_{\la}(\phi_n) >- \frac{\de}{\la_1} J(\phi_n)+\eta (J(\phi_n)), \quad  \forall n\in \N.  
\end{equation}
Thus, for $\phi_n$ with  $J(\phi_n) \geq 2,$ using the definition of $\eta,$ we have
$$
  C\ge \eta_{\la}(\phi_n) > - \frac{\de}{\la_1} J(\phi_n) + \frac{2 \de}{\la_1} (J(\phi_n) - 1) = \frac{\de}{\la_1} J(\phi_n) - \frac{2 \de }{\la_1}.
$$
Hence, $ J(\phi_n)\le  \displaystyle \max\Big\{2,\frac{\la_1 C}{\de}+2\Big\}. $
Now, we can use Lemma \ref{coercive} to obtain $C_1>0$ so that $\norm{\phi_n}_p\le C_1.$ 
Therefore, the sequence $(\phi_n)$ is bounded in $\wp$. 

\noi $(c)$ From \eqref{bounded below}, we have
$
\eta_{\la}(\phi) > - \frac{\de}{\la_1} J(\phi) + \eta (J(\phi)), \; \forall \phi\in \wp.
$
Therefore,
\begin{align*}
  \eta_{\la}(\phi) > \left\{\begin{array}{ll}
                                \frac{\de}{\la_1} J(\phi) - \frac{2 \de }{\la_1} > 0, & \; \text{ if }J(\phi) > 2; \\
                                - \frac{\de}{\la_1} J(\phi) + \eta (J(\phi)) \geq - \frac{2 \de }{\la_1}, &\;  \text{ if }J(\phi) \leq 2.
                              \end{array} 
                     \right.
\end{align*}
Thus $\eta_\la$ bounded below. 

\noi $(d)$ 
 By Lemma \ref{Poin}, there exists $m>0$ such that 
\begin{align}\label{eqn1}
    J(\phi) \geq m \norm{\phi}^p_p,\quad \forall\, \phi\in \wp \text{ with } G(\phi)>0.
\end{align}
 We choose $R_0=2(1+\frac{1}{m})$. Thus, for  $\phi\in \pa B_R(0)$ with $R>R_0,$  either $J(\phi)> 2$  or  $\norm{\phi}^p_p>  \frac{2}{m}.$
Notice that, 
$$
\big< \eta^{\prime}_{\la}(\phi), \phi \big> = p\Big(J(\phi) - \la G(\phi) + \eta'(J(\phi)) J(\phi)\Big).    
$$
Thus, using \eqref{bounded below}, we obtain
\begin{align*}
   \frac{1}{p} \big< \eta^{\prime}_{\la}(\phi), \phi \big> \geq - \frac{\de}{ \la_1} J(\phi) + {\eta}^{\prime}(J(\phi)) J(\phi).
\end{align*}
In particular, for $J(\phi)> 2$, we have
\begin{align*}
   \frac{1}{p} \big< \eta^{\prime}_{\la}(\phi), \phi \big> \geq   \frac{\de}{\la_1}J(\phi).
\end{align*}
 On the other hand, for  $J(\phi)\le 2,$ we have $\norm{\phi}^p_p>  \frac{2}{m}.$ Hence from \eqref{eqn1}, we conclude that $G(\phi)\le 0.$ Now from the part $(i)$ and $(ii)$ of proof of Proposition \ref{bounded below}, we get $\big< \eta^{\prime}_{\la}(\phi), \phi \big>>0.$
Therefore, $\eta^{\prime}_{\la}(\phi) \ne 0$ for $\phi\in \pa B_R(0)$ for any $R>R_0$. 
\end{proof}

Recall that a function $\phi \in \wp$ is a weak solution of \eqref{Steklov pertub}, if it satisfies the following weak formulation:
\begin{align*}
    \dis |\Gr \phi|^{p-2} \Gr \phi \cdot \Gr v - \la \dispa \left( g \abs{\phi}^{p-2} \phi v + f r(\phi)v \right) = 0, \quad \forall v \in \wp.
\end{align*}
Therefore, $\phi$ is a solution of \eqref{Steklov pertub} if and only if 
$$ 
\big< \left( J' - \la(G' + F) \right)(\phi), v \big>=0, \quad \forall v \in \wp.
$$

\begin{proposition}\label{class alpha 1}
The maps $J' - \la (G' + F)$  and  $J' - \la G'$  are well-defined maps from $\wp$ to its dual $(\wp)'.$ 
Moreover, these maps are bounded, demicontinuous and of class $\al(\wp)$. 
\end{proposition}

\begin{proof}
From Proposition \ref{G' cpt}, Proposition \ref{class alpha}, and Proposition \ref{compactmap1}, we obtain  $J' - \la (G' + F)$ and $J' - \la G'$ are well defined, bounded and demicontinuous. Since $J'$ is of class $\al(\wp)$ and $G', F$ are compact,  the maps  $J' - \la (G' + F)$ and $J' - \la G' $ are of class $\al(\wp)$.
\end{proof}

\begin{proposition}\label{index}
Let $g, \la_1$ be as given in Theorem \ref{Steklov existence}. Then  there exists $\de>0$ such that  for each $ \la \in (0, \la_1 + \de) \setminus \{ \la_1 \},$ $ind(J' - \la G' ,0)$ is well defined. Furthermore, \begin{enumerate}[(a)]
    \item $ind(J' - \la G',0) = 1 \;$ for $\la \in (0, \la_1)$,
    \item $ ind(J' - \la G',0) = -1 \;$ for $\la \in (\la_1, \la_1 + \de)$.
\end{enumerate}
\end{proposition}

\begin{proof}
Since $\la_1$ is an isolated eigenvalue of \eqref{Steklov weight}, there exists $\de > 0$ such that $ \la \in (0, \la_1 + \de)\setminus \{ \la_1 \}$  is not an eigenvalue of \eqref{Steklov weight}. Thus for $ \la \in (0, \la_1 + \de)\setminus \{ \la_1 \}$,  0 is the only solution of $J' - \la G'$ and hence $ind(J' - \la G' ,0)$ is well defined.
 
\noi $(a)$  For $\la \in (0, \la_1)$, from \eqref{bounded below}, we have 
\begin{align*}
      \big< \big(J' - \la G' \big)(\phi), \phi \big> = p(J(\phi) - \la G(\phi)) > 0, \quad \forall \phi \in \wp \setminus \{0 \}.
\end{align*}
Therefore, by  Proposition \ref{degree}, $deg(J' - \la G', \overline{B_{r}(0)}, 0) =1$ for every $r>0$. Thus 
$$
ind(J' - \la G',0) = \lim_{ r \rightarrow 0} deg(J' - \la G', \overline{B_{r}(0)}, 0)  = 1.
$$

\noindent $(b)$ In this case, we adapt a technique used in the proof of \cite[Theorem 4.1]{Drabek-Huang}. First, we compute $ind(\eta_\la',0)$.  Clearly, $0$ is a zero of $\eta'_{\la}$. If $\phi_0 \neq 0$ is a zero of $\eta_{\la}',$ then $ \frac{\la}{1 + \eta^{\prime}(J(\phi_0))}$ is an eigenvalue of \eqref{Steklov weight} and $ \phi_0 $ is a corresponding eigenfunction. Since $ 0 < \frac{\la}{1 + \eta^{\prime}(J(\phi_0))} < \la_1 + \de$, we must have  $\frac{\la}{1 + \eta^{\prime}(J(\phi_0))} = \la_1$ and $\phi_0 = c \phi_1$ for some $c \in \R,$ where $\phi_1 $ is the first eigenfunction of \eqref{Steklov weight} normalized as  $\int_{\pa \Om} g\phi^p_1=1$ and $\phi_1 > 0$ in $\overline{\Om}$. Notice that,
\begin{align*}
    \eta^{\prime}(J(\phi_0))  = \frac{\la}{\la_1} - 1 \in \left( 0, \frac{\de}{\la_1} \right).  
\end{align*}
Thus from  \eqref{df1}, we assert that $J(\phi_0) \in (1, 2).$ Moreover,  since $\eta^{\prime}$ is strictly increasing in $(1, 2)$ and the functional $J$ is even, there exists a unique $c > 0$ such that $\phi_0 = \pm c \phi_1.$ Conversely, if we choose $c>0$  such that $\eta^{\prime}(J(c\phi_1))  = \frac{\la}{\la_1} - 1,$ then $\pm c\phi_1$ is a zero of $\eta_\la'.$ Therefore, the map $ \eta^{\prime}_{\la}$ has precisely three zeros $-c\phi_1, 0, c\phi_1$. Now we will show that $ind(\eta^{\prime}_{\la},  \pm c\phi_1) = 1$. It is enough to prove $\pm c \phi_1$ are the minimizers for $\eta_{\la}$. From Lemma \ref{map properties}, the functional  $\eta_{\la}$ is coercive, weak lowersemicontinuous and bounded below. Thus $\eta_{\la}$ admits a minimizer. Notice that, $\eta_{\la}(t\phi_1)=(\la_1-\la)t^pG(\phi_1)+\eta(t^pJ(\phi_1))$ and hence  $\eta_{\la}(t\phi_1)<0$ for sufficiently small $t>0.$ Thus $0$ is not a minimizer and hence $\pm c \phi_1$ are the only minimizers of $\eta_{\la} $. Therefore, by Proposition \ref{degree}, we get
\begin{align}\label{index1}
   ind(\eta^{\prime}_{\la}, \pm c\phi_1) = 1.    
\end{align}
For $R_0$ as given in Lemma \ref{map properties}, we choose $R>R_0,$ so that $\pm c \phi_1 \in B_{R}(0)$ and $\big< \eta^{\prime}_{\la}(\phi), \phi \big>>0$ for $\phi\in \pa B_R(0)$.  By Proposition \ref{degree}, $deg(\eta^{\prime}_{\la},  \overline{B_{R}(0)}, 0) = 1.$ Thus by the additivity of degree (Proposition \ref{degree}) and from \eqref{index1}, we obtain $deg(\eta^{\prime}_{\la}, \overline{B_{r}(0)}, 0) = -1$ for  sufficiently small $r > 0$. Since $\eta^{\prime}_{\la} = J' - \la G'$ on  ${B_{r}(0)}$ for $r<1$, we conclude that $ind(J' - \la G',0) = -1.$ 
\end{proof}

\begin{Lemma}\label{homotopy invariance}
Let $\la_1$ be given as in Theorem \ref{Steklov existence}. Then for $ \la \in (0, \la_1 + \de) \setminus \{ \la_1 \}$, $ ind(J' - \la (G' + F), 0) = ind(J' - \la G' ,0).$ 
\end{Lemma}

\begin{proof}
For $\la \in (0, \la_1 + \de) \setminus \{ \la_1 \},$ define $H_{\la}: \wp \times [0,1] \rightarrow (\wp)'$ as $$H_{\la}(\phi, t) = J'(\phi) - \la G' (\phi) - \la t F(\phi).$$ Clearly, $H_{\la}(.,0) = J' - \la G'$ and $H_{\la}(.,1) = J' - \la (G' + F).$ From Proposition \ref{class alpha 1}, for each $t \in [0,1]$, $H_{\la}(\cdot, t)$ is bounded, demicontinuous and of class $\al(\wp)$. We prove the existence of a sufficiently small $r > 0$ such that  for each $t \in [0,1], \; H_{\la}(.,t)$ does not vanish in $\overline{B_r(0)} \setminus \{ 0 \}.$ On the contrary, assume that no such $r$ exists. 
Then for any $r > 0$, there exists $t_r \in [0,1]$ and $\phi_r \in \wp \setminus \{ 0 \}$ such that  $\norm{\phi_r}_{\wp} \leq r$ and $H_{\la}(\phi_r,t_r) = 0.$  In particular, for a sequence of positive numbers $(r_n)$ converging to 0, there exist a sequence $t_n\in [0,1]$ and a sequence $\phi_n \in \wp \setminus \{ 0 \}$  such that  $\norm{\phi_n}_{\wp} \leq r_n$  and 
\begin{align}\label{homotopy}
   J'(\phi_n) -\la G'(\phi_n) - \la t_n F(\phi_n) = 0.   
\end{align}
If we set $v_n = \phi_n {\norm{\phi_n}^{-1}_{\wp}},$ then  $\norm{v_n}_{\wp} = 1$ and  hence admits a subsequence $(v_{n_k})$ such that $v_{n_k} \rightharpoonup v$ in $\wp$. From \eqref{homotopy} we also have
\begin{align*}
    \big<  J'(v_{n_k}) - \la G'(v_{n_k}), v_{n_k} - v \big>  = \la t_{n_k} \left< \frac{F(\phi_{n_k})}{\norm{\phi_{n_k}}^{p-1}_{\wp}}, v_{n_k} - v \right>.
\end{align*}
 By  Proposition \ref{Growth}, the right hand side of the above inequality goes to zero as $k\ra \infty.$ Therefore, 
$$ 
\lim_{k \rightarrow \infty} \big< J'(v_{n_k}) - \la G'(v_{n_k}), v_{n_k} - v \big> = 0. 
$$
Now, since $J' - \la G'$ is of class $\al(\wp)$ (Proposition \ref{class alpha 1}),  we get $v_{n_k} \rightarrow v$ as $k \rightarrow \infty$. Thus  using \eqref{homotopy}, we deduce that
$J'(v) - \la G'(v) = 0$ and $\norm{v}_{\wp} = 1.$
A contradiction, as $\la \in (0, \la_1 + \de) \setminus \{ \la_1 \}$ is not an eigenvalue of \eqref{Steklov weight}. Therefore,  there exists $R > 0$ such that $H_\la(.,t)$ does not vanish in  $\overline{B_R(0)} \setminus \{ 0 \}.$ Thus 0 is an isolated zero of $H(.,t)$ for any $t\in [0,1]$. Hence by homotopy invariance of degree (Propostion \ref{degree}), we obtain
\begin{align}\label{index difference}
 ind(J' - \la (G' + F),0) = ind(J' - \la G' ,0)=\left\{\begin{array}{ll}
                                                          1, &\; \text{for}  \; \la \in (0, \la_1); \\ 
                                                         -1, &\; \text{for} \; \la \in (\la_1, \la_1 + \de).
                                                        \end{array}\right. 
\end{align}
\end{proof}

The following theorem gives a sufficient condition  \cite[Theorem 7.5, Page-61]{Skrypnik} under which $\la_1$ is a bifurcation point of \eqref{Steklov pertub}.

\begin{theorem}\label{Bifurcation}
Let $\la_1$ be given as in Theorem \ref{Steklov existence} and $g, r, f$ be given as in Theorem \ref{bifur}. Let
\begin{align*}
    \overline{i}^{\pm} =  \uplim_{\la \rightarrow \la_1 \pm 0} ind(J' - \la (G' + F),0); \quad  \underline{i}^{\pm} =  \lowlim_{\la \rightarrow \la_1 \pm 0} ind(J' - \la (G' + F),0). 
\end{align*}
If at least two of the numbers $\overline{i}^+, \underline{i}^+, \overline{i}^-,  \underline{i}^-, ind(J' - \la (G' + F),0) $ are distinct, then $\la_1$ is a bifurcation point of \eqref{Steklov pertub}. 
\end{theorem}

\begin{theorem}\label{bifurcation}
Let $\la_1$ be given as in Theorem \ref{Steklov existence} and $g, r, f$ be given as in Theorem \ref{bifur}. Then $\la_1$ is a bifurcation point of \eqref{Steklov pertub}.
\end{theorem}

\begin{proof}
From Proposition \ref{index} and Lemma \ref{homotopy invariance}, we have
\begin{align*}
 ind(J' - \la (G' + F),0) = \left\{\begin{array}{ll}
                                      1, &\; \text{for} \; \la \in (0, \la_1); \\ 
                                     -1, &\; \text{for} \; \la \in (\la_1, \la_1 + \delta).
                                     \end{array} 
                             \right.
\end{align*}
Therefore,
\begin{align*}
    \overline{i}^{+} =  \uplim_{\la \rightarrow \la_1 + 0}  ind(J' - \la (G' + F),0) = -1; \quad  \underline{i}^{-} =  \lowlim_{\la \rightarrow \la_1 - 0}  ind(J' - \la (G' + F),0) = 1.
\end{align*}
Thus, by Theorem \ref{Bifurcation}, $\la_1$ is a bifurcation point of \eqref{Steklov pertub}.
\end{proof}

The following lemma is proved as a part of \cite[Theorem 1.3]{Rabinowitz}.

\begin{Lemma}\label{ls and p}
Let $ r,g$ and $f$  be given as in Theorem \ref{bifur}. For $\la \in \mathbb{R},$ define
$$ r(\la) = \inf \left\{ \norm{\phi}_{\wp} > 0:  (J' - \la (G' + F))(\phi) = 0 \right\}.$$
Then $r$ is lower semicontinuous. Further more, if $\la$ is not an eigenvalue of   \eqref{Steklov weight}, then $r(\la)> 0$.

\end{Lemma}

\begin{proof}
 {\it $r$ is lower semicontinuous}:  Let $(\la_n)$ be a sequence in $\R^+$ such that $\la_n \rightarrow \la$. Without loss of generality we assume that $r(\la_n)$ is finite.  Now by definition of $r$, there exists $\phi_{n} \in \wp \setminus \{ 0 \}$ such that $\norm{\phi_{n}}_{\wp} < r(\la_{n}) + \frac{1}{n}$ and $(J' - \la_{{n}} (G' + F))(\phi_{n}) = 0.$ Since $(\phi_{n})$ is bounded, up to a subsequence $\phi_{n} \rightharpoonup \phi$ in $\wp$. Now by writing
\begin{align*}\label{p1}
     (J' - \la (G' + F))(\phi_n) = (J' - \la_n (G' + F))(\phi_{n}) + (\la_{n} - \la )(G' + F)(\phi_{n}),
\end{align*}
we observe that $\lim_{n \rightarrow \infty} \big< (J' - \la (G' + F))(\phi_{n}), \phi_{n} - \phi  \big> = 0$. As $J' - \la (G' + F)$ is of class $\al(\wp)$ (Proposition \ref{class alpha 1}), we get $\phi_{{n}} \rightarrow \phi$ in $\wp$. Therefore,
\begin{equation}\label{eq:class}
    (J' - \la (G' + F))(\phi) = 0
\end{equation}
We claim that $\phi\neq 0$. If not, then $\norm{\phi_{n}}_{\wp} \rightarrow 0,$ as $n \rightarrow \infty.$ Set $v_n = \phi_n {\norm{\phi_n}^{-1}_{\wp}}$. Then $v_n \rightharpoonup v$ in $\wp$ and (by the similar arguments as in the proof of Lemma \ref{homotopy invariance}) $v$ must be an eigenfunction corresponding to $\la.$
A contradiction and hence $\phi\neq 0.$ Thus, 
\begin{align*}
    r(\la) \le \norm{\phi}_{\wp} = \lim_{n \rightarrow \infty} \norm{\phi_n}_{\wp} \le \lim_{n \rightarrow \infty} \left( r(\la_n) + \frac{1}{n} \right)  = \lim_{n \rightarrow \infty} r(\la_n).
\end{align*}
{\it $r$ is positive}: Suppose $r(\la) = 0$ for some $\la$.  Then there exists a sequence $(\phi_n) \in \wp \setminus \{ 0 \}$ such that $\norm{\phi_n}_{\wp} < \frac{1}{n}$ and $(J' - \la (G' + F))(\phi_n) = 0.$ Set $v_n = \phi_n {\norm{\phi_n}^{-1}_{\wp}}.$ Then  $\norm{v_n}_{\wp} = 1$ and $v_n \rightharpoonup v$ in $\wp.$ Now using the similar arguments as in Lemma \ref{homotopy invariance}, we obtain 
\begin{align*}
    J'(v) - \la G'(v) = 0, \ \ \text{where} \; \norm{v}_{\wp} = 1.
\end{align*}
Thus $\la$ must be an eigenvalue of \eqref{Steklov weight}. Therefore, $r(\la)>0$, if $\la$ is not an eigenvalue of \eqref{Steklov weight}.
\end{proof}

\begin{remark}
If $(\la, 0)$ is a bifurcation point of \eqref{Steklov pertub}, then $r(\la) =  0$ and hence from Lemma \ref{ls and p}, $\la$ must be an eigenvalue of \eqref{Steklov weight}. Thus for the existence of a bifurcation point $(\la, 0)$ of \eqref{Steklov pertub}, it is necessary that $\la$ is an eigenvalue of \eqref{Steklov weight}.  
\end{remark}

In the next proposition we prove a generalized homotopy invariance property for the maps $J' - \la(G' + F).$ A similar result for  Leray-Schauder degree is obtained in \cite{Leray}. For a  set $U$ in $[a,b] \times \wp$, let $U_{\la}= \left\{ \phi \in \wp: (\la, \phi) \in U \right\}$ and $\pa U_{\la} = \left\{ \phi \in \wp: (\la, \phi) \in \pa U \right\}$. 

\begin{proposition}\label{gen homotopy}
Let  $U$ be a bounded  open set in  $[a,b] \times \wp$. If  $(J' - \la(G' + F))(\phi) \neq 0$ for every $\phi \in \pa U_{\la},$ then $deg(J' - \la(G' + F), U_{\la}, 0) = C,$  $\forall\, \la \in [a,b].$
\end{proposition}

\begin{proof}
It is enough to show that $deg(J' - \la(G' + F), U_{\la}, 0)$ is locally constant on $[a,b].$ Then the proof will follow from the connectedness of $[a,b]$ and the continuity of the degree. For each $\la\in [a,b],$
consider the set $N_{\la} = \left\{ \phi \in U_{\la} : (J' - \la(G' + F))(\phi) = 0 \right\}.$
For $\la_0 \in [a,b]$, let  $I_0 \subset [a,b]$ be a neighbourhood of $\la_0$ and let $V_0$ be an open set such that $N_{\la_0} \subset V_0 \subset \overline{V_{0}} \subset U_{\la_0}$ and $I_0 \times V_0 \subset U.$  We claim that there exists $$I_1 \subset I_0 \text{ such that } \la_0 \in I_1 \text{ and   } N_{\la} \subset V_0, \; \forall \la \in I_1.$$ If  not, then there exists a sequence $(\la_n,\phi_n)$ in $U$ such that $\phi_n \in N_{\la_n} \setminus V_0$ and $\la_n\ra \la_0.$ As $(\phi_n)$ is bounded in $\wp$, $\phi_n \rightharpoonup \phi$ for some  $\phi\in\wp$. Now following the steps that yield \eqref{eq:class},  we get  $\phi_n\ra \phi$ in $\wp$ and      
$(J' - \la_0 (G' + F))(\phi) = 0.$ Since  $\phi\in \overline{U_\la}$ and $J' - \la_0 (G' + F)$ is not vanishing on $\partial U_\la,$ we conclude $\phi \in U_\la$. Thus $\phi\in N_{\la_0},$  a contradiction since $\phi\not\in  V_0.$ Therefore, our claim must be true. Now consider the homotopy, $H : I_1 \times V_0 \rightarrow (\wp)'$  defined as $H(\la, \phi) = (J' - \la (G' + F))(\phi).$ By construction, for every $\la\in I_1,$ $H(\la,.)$  does not vanish on $\pa V_0.$ Thus by the classical homotopy invariance of degree (Proposition \ref{degree}),
$deg(H(\la, \cdot), V_0, 0) = C, \; \forall\, \la \in I_1.$ Since $H(\la, \phi) \neq 0$ in $U_{\la} \setminus V_0$, by the additivity of degree, we obtain  $deg(H(\la, \cdot), U_{\la}, 0) =  C, \; \forall\, \la \in I_1.$ 
\end{proof}

\noi{\textbf {Proof of Theorem \ref{bifur}}}: We adapt the technique used in the proof of \cite[Theorem 1.3]{Rabinowitz}.  Recall that  $\S \subset \R \times \wp$ is the set of all nontrivial solutions of $(J' - \la (G' + F))(\phi) = 0.$
Suppose there does not exist any continuum $\C \subset \S$ such that $(\la_1, 0) \in \C$ and  $\C$ is either unbounded, or  meets at $(\la, 0)$ where $\la$ is an eigenvalue of \eqref{Steklov weight} and $\la \neq \la_1$. Then by \cite[Lemma 1.2]{Rabinowitz},  there exists a bounded  open set $U \subset \R \times \wp$ containing $(\la_1,0)$  such that $\pa U \cap \S = \emptyset $ and  $\overline{U}\cap  \R \times\{0\}  =\overline{I}\times \{0\},$ where $I=(\la_1 - \de, \la_1 + \de)$ with $0<\de< \min\{\la_1,\la_2-\la_1\}.$ Thus $(\la \times \pa U_{\la}) \cap \S = \emptyset$ for every $\la\in \R$ and  $(\la, 0) \not \in \pa U$ for  $\la \in I$. In particular, $J' - \la (G' + F)$ does not vanish on $\partial U_\la$ for every $\la$ in $ I.$ Hence $deg((J' - \la (G' + F), U_{\la}, 0)$ is well defined and by homotopy invariance of degree (Proposition \ref{gen homotopy}), we  have
\begin{align}\label{d1}
    deg(J' - \la (G' + F), U_{\la}, 0) =  C, \ \ \text{for} \; \la \in I.
\end{align}
Next we compute $ind(J' - \la (G' + F),0)$ for $\la\in I$. Let $$d:=\text{dist}((-\infty,0]\cup [\la_2,\infty),\overline{U}).$$
 Since $\overline{U}\cap {\R\times\{0\}}=\overline{I}\times \{0\}$, we  observe that $d>0.$ Now set $$\rho(\la)=\left\{\begin{array}{ll}
   \frac{d}{2}, & \quad \text{ for } \la\in (-\infty,0]\cup [\la_2,\infty),\\
   \min\{1,\frac{1}{2}r(\la)\},  &  \quad \text{ for } \la\in (0,\la_2)\setminus\{\la_1\}. 
  \end{array}\right. 
$$
Thus using \ref{ls and p} we easily conclude that $\rho(\la)>0$ for each $\la\ne \la_1$ and $\overline{B_{\rho(\la)}} \setminus \{ 0 \}$ does not contain any solution of $J' - \la (G' + F).$   
Let $$I^* := \left\{\la: (\la,\phi)\in U \text{ for some } \phi\right\},\quad  \la^*:=\sup\{\la:\la\in I^*\},\quad  \la_*:=\inf\{\la:\la\in I^*\}$$ 
 For $\la\in (\la_1,\la^*],$ let $\rho = \inf \left\{ \rho(\mu): \mu \in [\la, \la^*] \right\}.$ By Lemma \eqref{ls and p}, we have $\rho > 0.$ Now consider the set $V= U \setminus [\la, \la^*] \times \overline{B_{\rho}}$. Observe that, $V$ is bounded and open in $[\la, \la^*] \times \wp$. Further more,  for each $\mu \in [\la, \la^*]$, $V_\mu=U_{\mu} \setminus \overline{B_{\rho}}$ and $(J' - \mu (G' + F))$ does not vanish on  $\pa V_\mu=\pa(U_{\mu} \setminus \overline{B_{\rho}})$. Therefore, by the homotopy invariance of degree (Proposition \ref{gen homotopy}) and noting that  $U_{\la^*} = \emptyset,$ we get 
$$ deg(J' - \la (G' + F), U_{\la} \setminus \overline{B_{\rho}}, 0) =  deg(J' - \mu (G' + F), U_{\la^*} \setminus \overline{B_{\rho}}, 0)=0.$$ 
Similarly,  for $\la \in [\la_*, \la_1)$ we  get $deg(J' - \la (G' + F), U_{\la} \setminus \overline{B_{\rho}}, 0) = 0$.
 Since $(J' - \la (G' + F))(\phi) \neq 0$ for $\phi \in B_{\rho(\la)} \setminus \overline{B_{\rho}},$ by the additivity of the degree we get $$deg(J' - \la (G' + F), U_{\la} \setminus \overline{B_{\rho(\la)}}, 0) = 0,\quad  \la \in [\la_*,\la^*]\setminus\{\la_1\}.$$ 
Again using the additivity of the degree, we conclude that
\begin{align*}
    deg(J' - \la (G' + F), U_{\la}, 0) = deg(J' - \la (G' + F), B_{\rho(\la)}, 0), \quad \forall \la \in I \setminus \{\la_1\}.
\end{align*}
Thus from \eqref{d1} we obtain
\begin{align*}
    ind(J' - \la(G' + F), 0) =  C, \quad \text{for} \; \la \in I\setminus \{\la_1\}.
\end{align*}
A contradiction to \eqref{index difference}. Thus there must exist  a continuous branch of non-trivial  solutions  from $(\la_1, 0)$ and is either  unbounded, or  meets at $(\la, 0)$ where $\la$ is an eigenvalue of \eqref{Steklov weight}.                                    \qed